\documentclass[a4paper,11pt]{amsart}

\usepackage{amsmath}           % AMS µÄÖ÷°ü
\usepackage{amssymb}           % AMS ·ûºÅ°ü£¬»á×Ô¶¯¼ÓÔØ amsfonts
\usepackage{mathabx}
\usepackage{latexsym}          % ¼¸¸ö¶îÍâµÄÌØÊâ·ûºÅ£¬°üÀ¨ \Box
\usepackage{mathrsfs}          % \mathscr ×ÖÌåÑùÊ½
\usepackage{verbatim}

\usepackage[all,2cell]{xy}
\usepackage{float}
\usepackage[colorlinks, linkcolor=blue,anchorcolor=Periwinkle,
citecolor=red,urlcolor=Emerald]{hyperref}
\usepackage[usenames,dvipsnames]{xcolor}
\usepackage{enumitem}

\usepackage{geometry,array} 
\usepackage{graphicx}
\usepackage{subfigure}
\usepackage{bookmark}
\usepackage{cleveref}

\usepackage{tikz-cd}  %%%%%nnew added
\usepackage{adjustbox} %%%%%new added
\usepackage{dynkin-diagrams}
\usepackage{color}

\usepackage{tikz}\usetikzlibrary{matrix}
\tikzcdset{arrow style=tikz, diagrams={>=stealth}}

\usetikzlibrary{decorations.markings}
\tikzset{->-/.style={decoration={  markings,  mark=at position #1 with
			{\arrow{>}}},postaction={decorate}}}
\tikzset{-<-/.style={decoration={  markings,  mark=at position #1 with
			{\arrow{<}}},postaction={decorate}}}
\usepackage[all]{xy}
\usetikzlibrary{arrows}
\usetikzlibrary{decorations.pathreplacing,decorations.pathmorphing,shadings,fadings,calc}
\usepackage{extarrows}
\newtheorem{thm}{Theorem}[section]
\newtheorem{lem}[thm]{Lemma}

\newcommand\<{\langle}
\renewcommand\>{\rangle}

\newcommand\iv[1]{\underline{#1}}

\newtheorem{prop}[thm]{Proposition}
\theoremstyle{definition}
\newtheorem{Def}[thm]{Definition}
\newtheorem{Ex}[thm]{Example}
\newtheorem{remark}[thm]{Remark}

\numberwithin{equation}{section}

\def\Irr{\operatorname{Irr}}

\newcommand\Br{\operatorname{Br}}

\newcommand\CBr{\operatorname{CT}}
\newcommand\Co{\operatorname{Co}}
\newcommand\per{\operatorname{per}}
\newcommand\pvd{\operatorname{pvd}}
\newcommand\fold{f:\Lambda\to\Delta}
\newcommand\Fold{f:\overrightarrow{\Lambda}\to\overrightarrow{\Delta}}

\newcommand{\EG}{\operatorname{EG}}       %exchange graph of heart (oriented)
\newcommand{\SEG}{\operatorname{SEG}}       %exchange graph of silting
\newcommand{\C}{\mathcal{C}}

\newcommand{\CEG}{\operatorname{CEG}}             %cluster tilting sets
\newcommand{\uCEG}{\underline{\operatorname{CEG}}}
\newcommand{\Cone}{\operatorname{Cone}}
\newcommand{\ad}{\operatorname{ad}}

\def\ceg{\operatorname{\mathcal{CEG}}}
\def\uceg{\underline{\ceg}}

\def\oi{\mathbf{i}}     %orbit of i
\def\oj{\mathbf{j}}
\def\n{\mathbf{n}}
\def\2{\mathbf{2}}
\def\1{\mathbf{1}}
\def\3{\mathbf{3}}
\def\4{\mathbf{4}}
\def\ol{\mathbf{l}}

\def\I{I_2(m)}
\def\Y{\mathbf{Y}}
\def\w{\mathbf{w}}
\newcommand{\W}{\mathbf{W}}

\def\x{\mathbf{x}}
\def\y{\mathbf{y}}

\newcommand{\D}{\operatorname{\mathcal{D}}}
\renewcommand{\k}{\mathbf{k}}
\def\wtq{\mathbf{Q}} %weighted quiver

\def\nn{node{$\bullet$}}
\def\h{\mathcal{H}}

\title{Cluster braid groups of Coxeter-Dynkin diagrams}
\author{Zhe Han, Ping He and Yu Qiu}
\address{Hz: School of Mathematics and Statistics
	Henan University 475004 Kaifeng China}
 \email{zhehan@vip.henu.edu.cn}
\address{Hp: Yanqi Lake Beijing Institute of Mathematical Sciences and Applications, 101408 Beijing, China}
\email{pinghe@bimsa.cn}
\address{Qy: Yau Mathematical Sciences Center and Department of Mathematical Sciences, Tsinghua University, 100084 Beijing, China. \&  Beijing Institute of Mathematical Sciences and Applications, Yanqi Lake, Beijing, China}
\email{yu.qiu@bath.edu}
\subjclass[2020]{52B05, 52B11, 20F36, 13F60, 05E10.}
\begin{document}
	
%=========================================================
\begin{abstract}
Cluster exchange groupoids are introduced by King-Qiu as an enhancement of cluster exchange graphs to study stability conditions and quadratic differentials. In this paper, we introduce the exchange groupoid for any finite Coxeter-Dynkin diagram $\Delta$ and show that the fundamental group of which is isomorphic to the corresponding braid group associated with $\Delta$.
\end{abstract}
\keywords{braid groups, cluster exchange groupoids, Coxeter-Dynkin diagrams, generalized associahedron}
%=========================================================

\maketitle

%=========================================================
\section{Introduction}
%=========================================================
Coxeter groups considered as reflection groups have a vibrant structure of geometry and algebra. Coxeter groups of finite types are classified Coxeter-Dynkin diagrams $\Delta$ \cite{Bo68}. A presentation of a Coxeter group is encoded in the Coxeter diagram. For each Coxeter group, the (Artin's) braid group is presented by the same generators $s_i$ and relations except $s^2_i=1$. Thus Coxeter groups are quotient groups of the corresponding braid groups. It is well-known that the braid groups are the fundamental groups of the space of regular orbits for which Coxeter groups are the corresponding complex reflection groups \cite{Br72, De, V83}.

The braid groups also appear in the theory of cluster algebras and stability conditions on triangulated categories \cite{ST, Q15, QW, Q16, KQ2}. In \cite{KQ2}, the braid group corresponding to a simply laced Dynkin diagram is realized as the cluster braid group, which is constructed using the combinatorial of the cluster category.
In this paper, we focus on the braid group associated to any finite type \emph{weighted graph/Coxeter graph} $\Delta$.
We give an alternative realization of this group by the corresponding cluster braid group.

Given an ordinary quiver with potential $(Q, W)$, let $\Lambda$ be the mutation equivalent class of $(Q,W)$. There is a Ginzburg dg algebra $\Gamma:=\Gamma(Q,W)$. The \emph{cluster category} $\C( \Lambda)\colon=\per(\Gamma)/\pvd(\Gamma)$ has an exchange graph $\uCEG(\Lambda)$ with vertices corresponding to cluster tilting objects and
unoriented edges corresponding to mutations between cluster tilting objects, cf. \cite{IY, K12}.
In many classes of examples, $\uCEG(\Lambda)$ can be decomposed into squares and pentagons
(i.e., \cite{FST} for the surface case and \cite{Q15} for the simply laced Dynkin case).
This is related to the pentagon identity of (quantum) dilogarithms, cf. \cite{K11}.
From the point view of tilting theory (i.e., simple HRS-tilting of hearts in $\pvd(\Gamma)$),
it is natural to consider the oriented version of $\uCEG(\Lambda)$ \cite{KQ} by replacing each edge with an oriented 2-cycle.
The cluster exchange groupoid $\ceg(\Lambda)$ is the quotient of the path groupoid of $\CEG(\Lambda)$ by the square and pentagon relations (and an extra hexagon relation in general).
The cluster braid group $\CBr_{\Lambda}(\Y)$ is defined to be the fundamental group of $\ceg(\Lambda)$ based at a vertex $\Y$ in $\CEG(\Lambda)$ (\Cref{def:ceg-wtq}).
The generators of which are 2-cycles (called local twists) mentioned above.

Cluster categories of simply laced (ADE) Dynkin type correspond to finite root systems and thus to finite crystallographic Coxeter groups \cite{FZ2}.
For a Coxeter diagram of type BCFG, one could define the cluster category using folding techniques,
cf. \cite{CQ}.
However, for finite non-crystallographic Coxeter graphs (i.e., of types H and I),
there is no automorphism of the diagram inducing the embedding of the corresponding root systems, cf. \cite{Dy, Lu}.
Thus there is no cluster category and cluster exchange graph of these types yet, from a categorical point of view.
In papers \cite{DT, DT2}, the authors extend the cluster theory to the quivers of types H and I
by mutations of matrices with real entries.
Their construction is based on the weighted folding technique, cf. \cite{Cr} for more details.
In this paper, we introduce and study cluster exchange groupoid $\ceg$ for all finite-type Coxeter-Dynkin diagrams.

Given a weighted graph $\overrightarrow \Delta$, one could construct a simplicial map $f\colon \overrightarrow \Lambda\to \overrightarrow \Delta$ from a simply laced quiver $\overrightarrow \Lambda$ which is called a weighted folding (see Definition \ref{def:fold}). For each finite weighted folding $f: \overrightarrow \Lambda\to \overrightarrow \Delta$, we define the vertex set of $\CEG(\Delta)$ as a subset of $\CEG(\Lambda)$. More precisely, each vertex is a cluster tilting object (CTO) $\Y$ in $\C(\Lambda)$, which is compatible with weighted folding, and hence is called a weighted CTO. The edges of $\CEG(\Delta)$ correspond to paths in $\ceg(\Lambda)$. In the finite type case, a key point is that the definition of $\CEG(\Delta)$ is independent of the choice of $f$. To obtain $\ceg(\Delta)$, we need to add an additional $(m+2)$-gon relation for each edge with weight $m$, which generalizes the square and pentagon relations \cite{KQ2} in the simply laced case. Along the way, we also show that $\uCEG(\Delta)$ decomposes into various $(m+2)$-gons. The \emph{cluster braid group} $\CBr_\Delta(\Y)$ is the subgroup of the fundamental group  $\pi_1(\ceg(\Delta),\Y)$, which is generated by local twists $t_\oi$ (indexed by $\oi\in \Lambda_0$).

On the other hand, for a weighted cluster tilting object (weighted CTO for short) $\Y$, there is an associated weighted quiver with potential $(\wtq_\Y,\W_\Y)$. We define the associated braid group $\Br(\wtq_\Y,\W_\Y)$ given by an explicit presentation with generators $b_\oi$ corresponding to vertices $\oi\in \wtq_0$ of $\wtq$ and relations corresponding to edges and terms in the potential $\W$, similarly to \cite{QZy}.
Our main theorem is the following,
which generalizes \cite[Thm.~2.16]{KQ2} for the ADE case.

\begin{thm}[\Cref{thm:main}]
Let $\overrightarrow{\Delta}$ be a finite weighted quiver. For any weighted CTO $\Y$ in $\CEG(\Delta)$,
there is a natural isomorphism
\begin{gather*}
    \begin{array}{rcl}
       \Psi: \CBr_\Delta(\Y)  & \longrightarrow & \Br(\wtq_\Y,\W_\Y) \\
        t_\oi & \mapsto & b_\oi
    \end{array}
\end{gather*}
where $t_\oi$ (resp. $b_\oi$) is the generator of $\CBr_\Delta(\Y)$ (resp. $\Br(\wtq_\Y,\W_\Y)$) corresponding to $\oi\in\Delta_0$.
\end{thm}

The paper is organized as follows. In Section~2, we recall some basic notations and results on weighted graphs and Artin braid groups.
In Section~3, we introduce braid groups with finite presentations for weighted quivers with potential of types H and I and show that the presentation is compatible with the mutation.
In Section~4, we define the cluster exchange groupoid for a finite weighted quiver. We prove that
its fundamental group is isomorphic to the corresponding  Artin braid group.

%=========================================================
\subsection*{Acknowledgment}
%=========================================================
Zhe Han would like to thank Department of Mathematical Sciences of Tsinghua University for its support and hospitality during the visit in 2023. In the twin paper [QZx], the fusion-stable interpretation of cluster categories/exchange graphs is given. This work is inspired by the works of Duffield and Tumarkin.

Zhe Han is supported by the National Natural Science Foundation of China (No. 12001164). Yu Qiu is supported by the National Key R$\&$D Program of China (No.2020 YFA0713000) and National Natural Science Foundation of China (No.12031007 and No.12271279).

%=========================================================
\section{Artin's braid groups associated to weighted graphs}
%=========================================================

A \emph{weighted graph} (or \emph{Coxeter graph}) $\Delta$ is a graph $(\Delta_0,\Delta_1)$ with at most one edge $\epsilon_{ij}$ between any two vertices $i,j\in\Delta_0$ and a weight function
$\w:\Delta_1\to\mathbb{Z}_{\geq 2}$. For convenience, we use $\epsilon_{ij}=\emptyset$ to denote the empty edge with weight $\w(\epsilon_{ij})=2$. We will regard a usual graph as a weighted graph with trivial weights $\w(\epsilon_{ij})\equiv3$ for $\epsilon_{ij}\ne\emptyset$.

The map could be encoded in Coxeter matrix $M=(m_{ij}=\w(\epsilon_{ij}))$ satisfies $m_{ii}=2$ and $m_{ij}=m_{ji}$.

For each positive integer $k\geq 2$, there is a relation
\[\Br^k(a,b)\colon\;\underbrace{aba\cdots}_{k}=\underbrace{bab\cdots}_{k}.\]
Here, the composition is from left to right.
For simplicity, we write $\Co(a,b)=\Br^2(a,b)$ and $\Br(a,b)=\Br^3(a,b)$.
Throughout, we will use the conjugation notation
$a^b=b^{-1}ab$, and thus $\Br(a,b)$ is equivalent to $a^b=b^a$.

\begin{Def}
Let $\Delta$ be a weighted graph with $|\Delta_0|=n$. The \emph{(Artin) braid group} $\Br\Delta$ associated with $\Delta$ is a group with generators $b_1,\cdots,b_n$ and defined as the following
\[
    \Br\Delta=
        \< b_1, \ldots, b_n \mid \Br^{\w(\epsilon_{ij})}(b_i,b_j), i \ne j \>
\]

The \emph{Coxeter group} $W(\Delta)$ of $\Delta$ is defined to be the
quotient group of $\Br\Delta$ by the normal subgroup generated by $b^2_i, 1\le i\leq n$.
A weighted graph $\Delta$ is called \emph{finite} if its Coxeter group $W(\Delta)$ is a finite group.
\end{Def}
The list in \Cref{fig:CD} provides a complete and non-redundant enumeration of finite Coxeter graphs.
Note that types GHI can be simplified to $H_{3,4}$ and $I_2(m\ge5)$.
\begin{figure}[htpb]\centering
	\begin{gather*}\renewcommand{\arraystretch}{2}
		\begin{array}{llr}
			A_{n}\ (n\geq 1): \quad &
			\xymatrix{1 \ar@{-}[r]& 2 \ar@{-}[r]& \cdots \ar@{-}[r]&n-1\ar@{-}[r]& n } \\
			B_{n}=C_n\ (n\geq 2): \quad &
			\xymatrix{1 \ar@{-}[r]& 2 \ar@{-}[r]& \cdots\ar@{-}[r]&n-1 \ar@{-}[r]^{\quad 4}& n }         \\
			D_{n}\ (n\geq 4): \quad &
			\xymatrix@R=0.4pc{
		       1 \\
		       & 3 \ar@{-}[ul]\ar@{-}[dl] \ar@{-}[r]& 4 \ar@{-}[r]& \cdots \ar@{-}[r]&n-1\ar@{-}[r]& n\\
		      2\\} \\
			E_{6,7,8}: \quad &
			\xymatrix@R=1.5pc{ && 4 \\
				1 \ar@{-}[r]& 2 \ar@{-}[r]& 3 \ar@{-}[r]\ar@{-}[u]& 5 \ar@{-}[r]& 6
				\ar@{.}[r]& 7 \ar@{.}[r]& 8}    \vspace{1mm}\\
			F_4: \quad &
			\xymatrix@R=1.5pc{
				1 \ar@{-}[r]& 2 \ar@{-}[r]^4& 3 \ar@{-}[r]&4}\\
			G_2: \quad &
            \xymatrix{1\ar@{-}[r]^-{6}&2}\\
            H_{2,3,4}: \quad &\xymatrix@R=1.5pc{
				1 \ar@{-}[r]^5& 2 \ar@{.}[r]& 3\ar@{.}[r]&4}\\
			I_2(m)\ (m\geq 7): \quad & \xymatrix@R=1.5pc{
				1 \ar@{-}[r]^m& 2}
		\end{array}
	\end{gather*}
	\caption{The complete list of finite Coxeter graphs}
	\label{fig:CD}
\end{figure}

The non-simply laced weighted graphs can be obtained from the simply laced ones via the so-called \emph{weighted folding} operation defined as follows.
\begin{Def}\cite{Cr}\label{def:fold}
	Let $\Lambda$ and $\Delta$ be two weighted graphs with $\Lambda$ a simply laced graph. A simplicial map $\fold$ is called a \emph{weighted folding} if for each edge $\epsilon=\oi\overset{m}{-}\oj$ in $\Delta$ the restriction $f_\epsilon$ of $f$ to $f^{-1}(\epsilon)$ is of one of the following types.
\begin{itemize}
		\item $f_\epsilon$ is a \emph{k-fold trivial folding}, i.e., $f^{-1}(\epsilon)=\epsilon^{\bigsqcup k}, k\in\mathbb{Z}_{\ge 0}$ and $f_\epsilon$ maps each copy identically onto $\epsilon$;
		
		\item $f_\epsilon$ is a \emph{dihedral folding}, i.e., $m>3$ and $f^{-1}(\epsilon)$ is a bipartite finite Coxeter graph of Coxeter number $m$ such that $f^{-1}(\oi)$ and $f^{-1}(\oj)$ are the two parts of $f^{-1}(\epsilon)$;
		
		\item $f_\epsilon$ is a \emph{composite folding}, i.e., $f_\epsilon=f'_\epsilon g$ where $f'_\epsilon$ is a k-trivial folding onto $\epsilon$ and $g$ is a folding onto $\epsilon^{\coprod k}$ which restricts to trivial and dihedral foldings.
	\end{itemize}

    A weighted folding is called \emph{finite} if $\Delta$ is a finite weighted graph.
\end{Def}

\begin{remark}\label{rem:fld} Given a weighted graph $\Delta$,
there is a general construction of its folding $f\colon \Lambda\to \Delta$ in \cite[\S~6]{CP}.

We remark that an unfolded graph $\Lambda$ of a weighted graph is possibly non-connected. However, each connected component of the unfolded graph is also an unfolding of the weighted graph. Indeed, let $\fold$ be a weighted folding and $\Lambda$ be a disjoint union $\Lambda_1\cup\cdots\cup\Lambda_s$  of connected components $\Lambda_i,1\leq i\leq s$. Then for each $\epsilon\in\Delta$, $f^{-1}(\epsilon)=\coprod^s_{i=1}f^{-1}_i(\epsilon)$ is a disjoint union of finite Coxeter graphs, where $f_i$ is the restriction of $f$ on $\Lambda_i$. Since $f$ is a folding, we have that $f_i:f^{-1}_i(\epsilon)\to\epsilon$ is of one of the three types in Definition~\ref{def:fold}. So the claim follows.
In the rest of this paper, we always assume unfolded graphs $\Lambda$ are connected without mentioning them.
\end{remark}

For a weighted folding $f\colon \Lambda\to \Delta$, it induces a group homomorphism between corresponding braid groups
\begin{equation}\label{eq:cri}
    \begin{array}{rcl}
        \iota_f:\Br\Delta & \longrightarrow & \Br\Lambda\\
        b_\oi & \mapsto & \prod_{j\in f^{-1}(\oi)}b_j,
    \end{array}
\end{equation}
where $b_\oi$ (resp. $b_j,j\in f^{-1}(\oi)$) is the generator of $\Br\Delta$ (resp. $\Br\Lambda$) given by vertex $\oi\in\Delta_0$ (resp. $j\in\Lambda_0$).
It is conjectured that the map $\iota_f$ is injective \cite{CP} and one has the following.

\begin{thm}[{\cite[Prop.~4.3 and Thm.~1.4]{Cr}}]\label{thm:cri}
Let $\fold$ be a finite weighted folding. Then $\iota_f$ in ~\eqref{eq:cri} is injective.
\end{thm}

\begin{Ex}[Finite weighted foldings]\label{ex:fold}
Let $\Delta$ be a finite weighted graph. For each edge $\epsilon \in\Delta_1$ of weight $m$, there are only finitely many simply laced graphs of Coxeter number $m$. So there are only finitely many unfoldings of $\Delta$.

Finite weighted foldings of Dynkin diagrams are well-known.
\begin{equation*}
    \begin{array}{ccccc}
        A_{2n-1}: & \begin{tikzcd}[row sep=small,column sep=small]
           1^+\ar[r,no head]&\cdots\ar[r,no head]&(n-1)^+\ar[dr,no head] &\\
			&&&n\\
			1^-\ar[r,no head]&\cdots\ar[r,no head]&(n-1)^-\ar[ur,no head] &
        \end{tikzcd}
         & \overset{f}{\longrightarrow} & C_n: & \begin{tikzcd}[row sep=.2em,column sep=small] \1\ar[r,no head]&\cdots\ar[r,no head]&\mathbf{n-1}\ar[r, no head,"4"]&\mathbf{n}
         \end{tikzcd}\\
         D_{n+1}:&\begin{tikzcd}[row sep=small,column sep=small]
             &&&n^+\\
             1\ar[r,no head] &\cdots\ar[r,no head]&n-1\ar[ur,no head]\ar[dr,no head]&\\
            	& & & n^-
         \end{tikzcd}& \overset{f}{\longrightarrow} &B_n: &\begin{tikzcd}[row sep=small,column sep=small]
            	\1\ar[no head,r]&\cdots\ar[r,no head]&\mathbf{n-1}\ar[r,no head,"4"]&\mathbf{n}
        \end{tikzcd}\\
        E_6:&\begin{tikzcd}[row sep=small,column sep=small]
            4^+\ar[no head, r]& 3^+\ar[no head,dr] & & \\
			&&2\ar[no head,r]&1\\
			4^-\ar[no head,r]&3^-\ar[no head,ur]&&
        \end{tikzcd}& \overset{f}{\longrightarrow} &F_4:&\begin{tikzcd}[row sep=small,column sep=small]
			\mathbf{4}\ar[no head,r]&\mathbf{3}\ar[no head,r,"4"]&\mathbf{2} \ar[no head,r]&\mathbf{1}
		\end{tikzcd}\\
        D_4:&\begin{tikzcd}[row sep=small,column sep=small]
            &&1^+\\
			2\ar[no head, urr]\ar[no head,drr]\ar[no head,rr]&&1^\circ\\
			&&1^-
        \end{tikzcd}& \overset{f}{\longrightarrow} &G_2:&\begin{tikzcd}[row sep=small,column sep=small]
            \2\ar[r,no head,"6"]&\1
        \end{tikzcd}
    \end{array}
\end{equation*}

Finite weighted foldings $\fold$ of non-crystallographic Coxeter diagrams $\Delta$ are listed in the following.
\begin{itemize}
    \item The graph $\Lambda$ is of type $E_8$ and $\Delta$ is of type $H_4$ while \[f(i^\pm)=\oi, 1\leq i\leq 4,\text{ and } f(\epsilon_{st})=\epsilon_{f(s)f(t)}, s,t\in(E_8)_0.\]
	Note that the restriction of $f$ onto the subgraph $\{i^\pm|1\leq i\leq 3\}$ gives a folding from $D_6$ to $H_3$.
    \begin{figure}[htpb]\centering
		\begin{tikzcd}[row sep=.2em,column sep=scriptsize]
			\qquad\qquad4^+\ar[no head,r]& 3^+\ar[no head,r] & 2^+\ar[no head,r] & 1^-&&&&&\\
			\quad E_8:\qquad&&&&\xlongrightarrow{f}  &H_4:\;\;\mathbf{4}\ar[no head,r]&\mathbf{3}
                \ar[no head,r]&\mathbf{2} \ar[no head,r,"5"]&\mathbf{1}\\
			\qquad\qquad4^-\ar[no head,r]&3^-\ar[no head,r]&2^-\ar[no head,r]\ar[no head,uur]& 1^+&&&&&		
		\end{tikzcd}
	\end{figure}

    \item The graph  $\Lambda$ is simply laced Dynkin of Coxeter number $m$ and $\Delta$ is of type $\I$ while $f$ maps bullets in the first (resp. second) row of $\Lambda$ to $\1$ (resp. $\2$) and all edges to the unique edge in $\I$
    \begin{figure}[htpb]\centering
	\begin{tikzcd}[row sep=.05em, column sep=1pc]
		&& \circ & &\cdots& & \circ& &&&&\1\\
		A_{m-1}\colon&&&&&& &&\overset{f}{\longrightarrow} & I_2(m):&\\
		&\circ\ar[no head,uur]&&\circ\ar[no head,uur]\ar[no head,uul] && \circ
            \ar[no head,uul]\ar[no head,uur]&&\circ \ar[no head,uul]&&&&\2\ar[no head,uu,"m"]
    \end{tikzcd}\
	\begin{tikzcd}[row sep=.05em, column sep=1pc]
		&& \circ & &\cdots& & \circ& &&&&\1\\
		D_{\frac{m}{2}+1}\colon&&&&&& &&\overset{f}{\longrightarrow} & I_2(m):&\\
		&\circ\ar[no head,uur]&&\circ\ar[no head,uur]\ar[no head,uul] && \circ
            \ar[no head,uul]\ar[no head,uur]& \circ\ar[no head,uu] &\circ \ar[no head,uul]&&&&\2\ar[no head,uu,"m"]
    \end{tikzcd}
	\begin{tikzcd}[row sep=.05em, column sep=1pc]
		&& 1 & 3 &5& & 7& &&&&\1\\
		E_{6,7,8}\colon&&&&&& &&\overset{f}{\longrightarrow} & I_2(m):&\\
		&\;\;\;2\ar[no head,uur]&&4\ar[no head,uu]\ar[no head,uur]\ar[no head,uul] && 6
            \ar[no head,uul]\ar[no head,uur]&&8 \ar[no head,uul]&&&&\2\ar[no head,uu,"m"]
    \end{tikzcd}
    \end{figure}
\end{itemize}
\end{Ex}

%==========================================================
\section{Presentation of a weighted quiver with potential}
%==========================================================

An ordinary \emph{quiver} $Q$ is a quadruple $(Q_0,Q_1,s,t)$,
where $Q_0$ denotes the set of vertices, $Q_1$ denotes the set of arrows, and $s$ and $t$ are two maps that associate an arrow with its start and end points, respectively. A \emph{potential} $W$ of $Q$ is a linear combination (possibly infinite) of cyclic paths of $Q$, with the property that no two different cycles in $W$ with non-zero coefficient can be obtained from each other by rotations. The pair $(Q,W)$ is called a \emph{quiver with potential} (QP for short).

Mutation of QP is developed by Derksen-Weyman-Zelevinsky in \cite{DWZ}, which is an involution on a QP $(Q,W)$ and results in a new quiver with QP $(Q',W')$. We follow the formulations in \cite{DWZ}.

A \emph{weighted quiver} is an oriented weighted graph. A weighted quiver with potential (weighted QP for short) is a pair $(\wtq,\W)$ where $\wtq$ is a weighted quiver and $\W$ is a linear combination of cycles in $\wtq$, with the property that no two different cycles in $W$ with non-zero coefficient can be obtained from each other by rotations.

A map between $(Q,W)$ and $(\wtq,\W)$ is called a \emph{folding} if it is a folding between the underlying graphs that is compatible with the orientation, and sends $W$ to $\W$.

In the rest of this paper, when we have a finite weighted folding $\fold$ in \Cref{ex:fold}, we fix an orientation $\overrightarrow{\Lambda}$ (resp. $\overrightarrow{\Delta}$) of the finite graph $\Lambda$ (resp. $\Delta$), and
\begin{equation}\label{eq:f}
    f\colon \overrightarrow{\Lambda}\longrightarrow \overrightarrow{\Delta}
\end{equation}
represents the folding $(\overrightarrow{\Lambda},0)\to (\overrightarrow{\Delta},0)$.

For any $\oi\in\wtq_0$, we call $\mu_\oi=\prod_{i\in f^{-1}(\oi)}\mu_i$ a \emph{weighted mutation} with respect to $f$. We denote
\[\mu_{\oi_1,\cdots,\oi_n}:=\mu_{\oi_1}\ldots\mu_{\oi_n}\] for the sequence of repeated weighted mutations.
\begin{Def}\label{def:mu}
    For any mutation sequence $\mu_{\oi_1,\cdots,\oi_n}$, define
    \[\mu_{\oi_1,\cdots,\oi_n}(\overrightarrow{\Delta})=f\circ\mu_{\oi_1,\cdots,\oi_n}(\overrightarrow{\Lambda})\]
    called the \emph{mutation} of the weighted quiver $\overrightarrow{\Delta}$.
\end{Def}

\begin{lem}\label{lem:mu-well}
    The mutation of a weighted quiver is well-defined.
\end{lem}
\begin{proof}
    For $\Delta$ of types $B_n, C_n, F_4$ and $G_2$, it is obtained from $\Lambda$ by a group action, which is compatible with mutations, see \cite{CQ} for the categorical constructions. For $\Delta$ of types $H$ and $I$, the assertion follows by a direct check, or \cite{DT,DT2},
    which also alternatively follows from a categorical interpretation via fusion actions,
    cf. \cite{QZx}.
\end{proof}

Two weighted QPs are said to be \emph{mutation-equivalent} if a sequence of mutations relates them. A weighted QP is called \emph{finite} if it is mutation-equivalent to $\overrightarrow{\Delta}$ with $\Delta$ a finite weighted graph.
\begin{Def}\label{def:braid group}
	Let $(\wtq,\W)$ be a finite weighted QP. We define the associated braid group $\Br(\wtq,\W)$ by the following presentation.
    \begin{itemize}
        \item Generators: vertices of $\wtq$.
        \item Relations $\Br^m(b_\oi,b_\oj)$ if there is exactly one arrow between $\oi$ and $\oj$
        of weight $m$;
        \item Relations from potential terms:
        \begin{enumerate}[label=\Roman*.]
            \item $\Co(b_\2,b_\1^{b_{\ol} b_{\ol-1} \cdots b_\3})$ if there is a $l$-cycle, $l\geq3$, between $\1,\2,\cdots,\ol$ which contributes a term in $\W$ and has all arrows of weight 3, see Figure~\ref{fig:cycle-present}~I;
            \item $\Co(b_\1,b^{b_\2}_\3)$ if there is a 3-cycle between $\1,\2$ and $\3$ which contributes a term in $\W$ such that the two arrows incident to $\1$ are of weights $m$ for $m\in\{4,5\}$ and the arrow opposite to $\1$ is of weight 3, see Figure~\ref{fig:cycle-present}~II;
            \item $\Co(b_\3,b^{b_\1b_\4}_\2)$ if there is a 4-cycle between $\1,\2,\3$ and $\4$ which contributes a term in $\W$ such that the two arrows end at $1$ and $3$ are of weight $m$ for $m\in\{4,5\}$ while the two arrows start at $1$ and $3$ are of weight $3$, see Figure~\ref{fig:cycle-present}~III;
            \item $\Co(b_\1,b^{b_\3b_\2}_\2)$ and $\Br(b_\1,b^{b_\2}_\3)$ if there is a 3-cycle between $\1,\2$ and $\3$ which contributes a term in $\W$ and has all arrows of weight 5, see Figure~\ref{fig:cycle-present}~IV.
        \end{enumerate}
    \end{itemize}
\end{Def}
\begin{figure}\centering
\begin{tikzpicture}
    \draw(0,2)node{\begin{tikzcd}[row sep=3.4em,column sep=2em]
        \1\ar[rr,"3"]&&\2\ar[d,"3"]& &\1\ar[dr,"m"]& & \2\ar[rr,"m"]&&\3\ar[d,"3"] & &\1\ar[dr,"5"]&\\
        \ol\ar[r,no head]\ar[u,"3"]&\cdots &\3\ar[l,no head] & \3\ar[ur,"m"]&&\2\ar[ll,"3"] & \1\ar[u,"3"]&&\4\ar[ll,"m"] & \3\ar[ur,"5"]&&\2\ar[ll,"5"]
    \end{tikzcd}};
\draw(-5.6,0)node{\text{I}}(-1.6,0)node{\text{II}}(2.1,0)node{\text{III}}(5.8,0)node{\text{IV}};
\end{tikzpicture}
\caption{Type of potential terms}
\label{fig:cycle-present}
\end{figure}
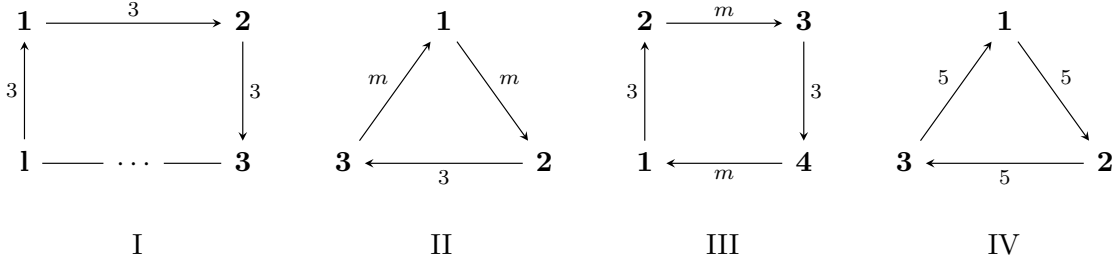

\begin{remark}
    Note that $\Br(\wtq,0)$ is exactly $\Br\wtq$ for any finite weighted quiver $\wtq$. Moreover, when $(\wtq,\W)$ mutation-equivalent to a Dynkin quiver, our definition coincides with \cite[Def.~10.1]{Q16}, \cite[Def.~2.2]{GM} and \cite[Def.~4.5]{HHLP}.

    We remark  that relations from potential terms have equivalent descriptions as the following
    (cf. Figure~\ref{fig:cycle-present}).
    \begin{itemize}
        \item For type I, we have the following equivalences of relations
        \[\begin{tikzcd}[row sep=small,column sep=huge]
            \Co(b_\1,b^{b_{\ol-\1}\cdots b_\2}_\ol)\ar[Leftrightarrow]{r}{(-)^{\iv{b_\2}}}& \Co(b^{\iv{b_\2}}_\1,b^{b_{\ol-\1}\cdots b_\3}_\ol)\ar[Leftrightarrow]{r}{\Br(b_\1,b_\2)}& \Co(b^{b_\1}_\2,b^{b_{\ol-\1}\cdots b_\3}_\ol)\\
             \ar[Leftrightarrow,shorten=5.3mm,xshift=3.5ex]{r}{\iv{b_\1}} & \Co(b_\2,b^{b_{\ol-\1}\cdots b_\3\iv{b_\1}}_\ol) \ar[Leftrightarrow]{r}{\Co(b_\3,b_\1)} & \Co(b_\2,b^{b_{\ol-\1}\cdots \iv{b_\1}b_\3}_\ol)\\
             \ar[Leftrightarrow,shorten=11mm,xshift=0ex]{r}{\Co(b_{\4},b_\1)} & \ldots           \ar[Leftrightarrow,shorten=5.8mm,xshift=3.5ex]{r}{\Co(b_{\ol-\1},b_\1)} & \Co(b_\2,b^{\iv{b_\1}b_{\ol-\1}\cdots b_\3}_\ol)\\
             \ar[Leftrightarrow,shorten=5.8mm,xshift=3.5ex]{r}{\Br(b_\1,b_\ol)} & \Co(b_\2,b^{b_\ol b_{\ol-\1}\cdots b_\3}_\1)
        \end{tikzcd}\]
        which are also equivalent to any
        $\Co(b_{\oi+1},b_{\oi}^{b_{\oi-1}\cdots b_{\oi+1}})$ (with the convention $\oi=\ol+\oi$ here), see also \cite[Lem.~10.2]{Q16} and \cite[Lem.~5.4]{QZy};

        \item For type II, we have that
        \[\begin{tikzcd}[row sep=small, column sep=huge]
            \Co(b_\1,b^{b_\2}_\3)\ar[Leftrightarrow]{r}{\Br(b_\2,b_\3)} &  \Co(b_\1,b^{\iv{b_\3}}_\2)\ar[Leftrightarrow]{r}{(-)^{b_\3}}& \Co(b_\2,b^{b_\3}_\1).
        \end{tikzcd}\]

        \item For type III, we have that
        \[\begin{tikzcd}[row sep=small,column sep=huge]
            \Co(b_\1,b^{b_\3b_\2}_\4)\ar[Leftrightarrow]{r}{\Br(b_\4,b_\3)} & \Co(b_\1,b^{\iv{b_\4}b_\2}_\3)\ar[Leftrightarrow]{r}{(-)^{\iv{b_\2}b_\4}} & \Co(b^{\iv{b_\2}b_\4}_\1,b_\3)\\
            \ar[Leftrightarrow,shorten=5mm,xshift=3.2ex]{r}{\Br(b_\1,b_\2)} & \Co(b^{b_\1b_\4}_\2,b_\3).
        \end{tikzcd}\]

        \item For type IV, we have
        \[\begin{tikzcd}[row sep=small,column sep=huge]
              \Co(b_\1,b^{b_\3b_\2}_\2)\ar[Leftrightarrow]{r}{\Br^5(b_\2,b_\3)} &  \Co(b_\1,b^{\iv{b_\2}\iv{b_\3}}_\3)\ar[Leftrightarrow]{r}{(-)^{b_\3b_\2}} & \Co(b^{b_\3b_\2}_\1,b_\3)\\
              \ar[Leftrightarrow,shorten=4.5mm,xshift=3.5ex]{r}{\Br(b_\2,b^{b_\3}_\1)} &  \Co(b^{\iv{b_\3}\iv{b_\1}b_\3}_\2,b_\3)\ar[Leftrightarrow]{r}{(-)^{\iv{b_\3}b_\1b_\3}} & \Co(b_\2, b^{b_\1b_\3}_\3).
        \end{tikzcd}\]
        On the other hand, we have
        \[\begin{tikzcd}[row sep=small,column sep=huge]
            \Co(b_\1,b^{b_\3b_\2}_\2)\ar[Leftrightarrow]{r}{\Br^5(b_\2,b_\3)} &  \Co(b_\1,b^{\iv{b_\2}\iv{b_\3}}_\3)\ar[Leftrightarrow]{r}{(-)^{b_\3b_\2}} & \Co(b^{b_\3b_\2}_\1,b_\3).
        \end{tikzcd}\]
        So
        \[\begin{tikzcd}[row sep=small,column sep=huge]
            \Br(b_\1,b^{b_\2}_\3)\ar[Leftrightarrow]{r}{\Br^5(b_\2,b_\3)} & \Br(b_\1,b^{\iv{b_\3}}_\2)\ar[Leftrightarrow]{r}{(-)^{b_\3}} & \Br(b^{b_\3}_\1,b^{\iv{b_\3}\iv{b_\2}}_\2)\\
            \ar[Leftrightarrow,shorten=4mm,xshift=3.2ex]{r}{(-)^{b_\2b_\3}} & \Br(b^{b_\3b_\2b_\3}_\1,b_\2)\ar[Leftrightarrow]{r}{\Co(b^{b_\3b_\2}_\1,b_\3)} & \Br(b^{b_\3b_\2}_\1,b_\2) \\
            \ar[Leftrightarrow,shorten=4.8mm,xshift=3ex]{r}{(-)^{\iv{b_\2}}} & \Br(b^{b_\3}_\1,b_\2).
        \end{tikzcd}\]
        A similar discussion shows that $\Co(b_\2,b^{b_\1b_\3}_\3)$ and $\Br(b_\2,b^{b_\3}_\1)$ are equivalent to $\Co(b_\3,b^{b_\2b_\1}_\1)$ and $\Br(b_\3,b^{b_\1}_\2)$.
    \end{itemize}
\end{remark}

Let $(\wtq,\W)$ be a finite weighted QP. For each $\oi\in\wtq_0$, denote by $\wtq'=\mu_\oi(\wtq)$ and by $b'_\oi,\oi\in\wtq_0$ the generators of $B_{\wtq'}$. Let
\begin{align*}
    \theta^\flat_\oi:\Br(\wtq,\W)\longrightarrow\Br(\wtq',\W')\\
    \theta^\sharp_\oi:\Br(\wtq',\W')\longrightarrow\Br(\wtq,\W)
\end{align*}
satisfying
\begin{align}\label{eq:flat}
    \theta^\flat_\oi(b_\oj)=\begin{cases}
	(b'_\oj)^{b'_\oi} &\text{if there is an arrow from $\oj$ to $\oi$ in $\wtq$,}\\
	b'_\oj & \text{otherwise;}
\end{cases}
\end{align}
\begin{align}\label{eq:sharp}
    \theta^\sharp_\oi(b'_\oj)=\begin{cases}
	(b_\oj)^{\iv{b_\oi}} &\text{if there is an arrow from $\oi$ to $\oj$ in $\wtq'$,}\\
	b_\oj & \text{otherwise.}
\end{cases}
\end{align}

\begin{thm}\label{thm:mu-present}
	For any finite weighted QP $(\wtq,\W)$, $\theta^\sharp_\oi$ and $\theta^\flat_\oi,\oi\in\wtq_0$ are mutually inverse group isomorphisms.
\end{thm}
\begin{proof}
    The assertion follows by \cite{GM} (see also \cite{QZy}) for simply laced Dynkin quivers and by \cite[Thm.~6.1]{HHLP} for weighted quivers of types $B_n=C_n,F_4,G_2$. We only prove weighted QPs of types $H$ and $\I$. Formulas ~\eqref{eq:flat} and ~\eqref{eq:sharp} define two homomorphisms $\theta^\sharp_\oi$ and $\theta^\flat_\oi$ between the groups that are freely generated by the vertices of $\wtq$ and $\wtq'$, respectively. Moreover, $\theta^\sharp_\oi$ is the composition of $\theta^\flat_\oi$ with the conjugation by $b_\oi$, and $\theta^\sharp_\oi$ and $\theta^\flat_\oi$ are mutually inverse. So we only need to show $\theta^\sharp_\oi$ preserves the relations in the presentation of $\Br(\wtq',\W')$.

	For any vertex $j$, denote by $t_\oj=\theta^\sharp_\oi(b'_\oj)$. Then we have
    \[\Br^m(b'_\oi,b'_\oj)\xLongleftrightarrow{\theta^\sharp_\oi} \Br^m(t_\oi,t_\oj)\xLongleftrightarrow{~\eqref{eq:sharp}}
    \begin{cases}
    \Br^m(b_\oi,(b_\oj)^{\iv{b_\oi}})\xLongleftrightarrow{(-)^{\iv{b_\oi}}} \Br^m(b_\oi,b_\oj) & \text{if }\oi\to\oj\textbf{ in }\wtq',\\
        \Br^m(b_\oi,b_\oj) & \text{otherwise.}
    \end{cases}\]
    where $m$ is the weight of the edge in $(\wtq',\W')$ between $\oi$ and $\oj$. A similar discussion shows that $\Br^k(t_\oj,t_\ol)$ holds for any $\oj,\ol\neq\oi$ such that there is no cycles between $\oi,\oj,\ol$ in $(\wtq,\W)$ and $(\wtq',\W')$, where $k$ is the weight of the edge in $(\wtq',\W')$ between $\oj$ and $\ol$.

    It remains to show that $\theta^\sharp_\oi$ preserves the relations of cycles of $\Br(\wtq',\W')$. A direct calculation shows that, up to the opposite of weighted QPs, all possible cycles and related mutations are listed below. The images of the relations in $\Br(\wtq',\W')$ under $\theta^\sharp_\oi$ are equivalent to the relations in $\Br(\wtq,\W)$, where $t_\oj=\theta^\sharp_\oi(b'_\oj)$. In each case of the list, the left weight QP is (the full subquiver of) $(\wtq',\W')$ and the right weighted QP is (the corresponding full subquiver of) $(\wtq,\W)$, and an equivalence may use equivalences above it.
\begin{gather*}
\begin{tikzpicture}[scale=0.5,
  arrow/.style={->,>=stealth},
  equalto/.style={double,double distance=2pt},
  mapto/.style={|->}]
\node (x2) at (0,2){};
\node (x1) at (2,-1){};
\node (x3) at (-2,-1){};
\node at (x1){$\1$};
\node at (x2){$\2$};
\node at (x3){$\3$};
\foreach \n/\m in {3/2,2/1}{\draw[arrow] (x\n) to (x\m);}
\draw(-1.3,.1)node[above]{\tiny{$3$}}(1.3,.1)node[above]{\tiny{$5$}};
\draw(4,1)node[above]{$\mu_\2$}node[below]{\tiny{$t_\1=b_\1^{\iv{b_\2}}$}};
\draw(4,-.4)node{\tiny{$t_\2=b_\2$}}(4,-1)node{\tiny{$t_\3=b_\3$}};
\draw[arrow,decorate, decoration={snake}](3,1)to(5,1);
\node (x2) at (0+8,2){};
\node (x1) at (2+8,-1){};
\node (x3) at (-2+8,-1){};
  \node at (x1){$\1$};
  \node at (x2){$\2$};
  \node at (x3){$\3$};
  \foreach \n/\m in {1/2,2/3,3/1}
    {\draw[arrow]
     (x\n) to (x\m);}
     \draw(6.7,.1)node[above]{\tiny{$3$}}(9.4,.1)node[above]{\tiny{$5$}}(8,-1.3)node{\tiny{5}};
\draw(17,.5)node{
$\begin{array}{ccl}
\Co(t_\1,t_\3)&\Longleftrightarrow&\Co(b_\1,b_\3^{b_\2})\\
\Br^5(t_\1,t_\2)&\Longleftrightarrow&\Br^5(b_\1,b_\2)\\
\Br(t_\2,t_\3)&\Longleftrightarrow&\Br(b_\2,b_\3)
\end{array}$
};
\end{tikzpicture}
\end{gather*}
\begin{gather*}
\begin{tikzpicture}[scale=0.5,
  arrow/.style={->,>=stealth},
  equalto/.style={double,double distance=2pt},
  mapto/.style={|->}]
\node (x) at (-3,2){};
\node (x2) at (0,2){};
\node (x1) at (2,-1){};
\node (x3) at (-2,-1){};
\node at (x1){$\1$};
\node at (x2){$\2$};
\node at (x3){$\3$};
\foreach \n/\m in {3/2,2/1,1/3}{\draw[arrow] (x\n) to (x\m);}
\draw(-1.3,.1)node[above]{\tiny{$5$}}(1.3,.1)node[above]{\tiny{$5$}}(0,-1.3)node{\tiny{$3$}};
\draw(4,1)node[above]{$\mu_\2$}node[below]{\tiny{$t_\1=b_\1^{\iv{b_\2}}$}};
\draw(4,-.4)node{\tiny{$t_\2=b_\2$}}(4,-1)node{\tiny{$t_\3=b_\3$}};
\draw[arrow,decorate, decoration={snake}](3,1)to(5,1);
\node (x2) at (0+8,2){};
\node (x1) at (2+8,-1){};
\node (x3) at (-2+8,-1){};
  \node at (x1){$\1$};
  \node at (x2){$\2$};
  \node at (x3){$\3$};
  \foreach \n/\m in {1/2,2/3,3/1}
    {\draw[arrow]
     (x\n) to (x\m);}
\draw(6.7,.1)node[above]{\tiny{$5$}}(9.4,.1)node[above]{\tiny{$5$}}(8,-1.3)node{\tiny{5}};
\draw(17.5,.5)node{
$\begin{array}{ccl}
\Co(t_\1,t^{t_\3}_\2)&\Longleftrightarrow&\Co(b_\1,b_\2^{b_\3b_\2})\\
\Br(t_\1,t_\3)&\Longleftrightarrow&\Br(b_\1,b^{b_\2}_\3)\\
\Br^5(t_\1,t_\2)&\Longleftrightarrow&\Br^5(b_\1,b_\2)\\
\end{array}$
};
\end{tikzpicture}
\end{gather*}

\begin{gather*}
\begin{tikzpicture}[yscale=0.6,xscale=0.5,
  arrow/.style={->,>=stealth},
  equalto/.style={double,double distance=2pt},
  mapto/.style={|->}]
\node (x2) at (-4,2){};
\node (x3) at (0,2){};
\node (x4) at (0,-1){};
\node (x1) at (-4,-1){};
  \node at (x1){$\1$};
  \node at (x2){$\2$};
  \node at (x3){$\3$};
  \node at (x4){$\4$};
  \foreach \n/\m in {2/1,3/2,1/3,4/1,3/4}
    {\draw[arrow]
     (x\n) to (x\m);}
\draw(-4,.5)node[left]{\tiny{$3$}}(0,.5)node[right]{\tiny{$3$}}(-2,-1)node[below]{\tiny{$5$}}(-2,2)node[above]{\tiny{$5$}}(-2.2,.7)node{\tiny{$5$}};
\draw(2.5,1.5)node[above]{$\mu_\2$}node[below]{\tiny{$t_\1=b_\1^{\iv{b_\2}}$}};
\draw(2.5,.1)node{\tiny{$t_\2=b_\2$}}(2.5,-.5)node{\tiny{$t_\3=b_\3$}}(2.5,-1.1)node{\tiny{$t_\4=b_\4$}};
\draw[arrow,decorate, decoration={snake}](1.5,1.5)to(3.5,1.5);
\node (x2) at (5,2){};
\node (x3) at (9,2){};
\node (x4) at (9,-1){};
\node (x1) at (5,-1){};
\node at (x1){$\1$};
\node at (x2){$\2$};
\node at (x3){$\3$};
\node at (x4){$\4$};
\foreach \n/\m in {1/2,2/3,3/4,4/1}{\draw[arrow] (x\n) to (x\m);}
\draw(5,.5)node[left]{\tiny{$3$}}(9,.5)node[right]{\tiny{$3$}}(7,-1)node[below]{\tiny{$5$}}(7,2)node[above]{\tiny{$5$}};
\draw(16,.5)node{
$\begin{array}{ccl}
%&\Co(b'_\2,b'_\4)&\\
\Br(t_\1,t_\2)&\Longleftrightarrow&\Br(b_\1,b_\2)\\
\Co(t_\3,t^{t_\2}_\1)&\Longleftrightarrow&\Co(b_\1,b_\3)\\
\Br(t_\3,t_\4)&\Longleftrightarrow&\Br(b_\3,b_\4)\\
\Co(t_\3,t^{t_\4}_\1)&\Longleftrightarrow&\Co(b_\3,b_\2^{b_\1b_\4})\\
\Br^5(t_\1,t_\3)&\Longleftrightarrow&\Br^5(b_\2,b_\3)
\end{array}$
};
\end{tikzpicture}
\end{gather*}

\end{proof}

%====================================================
\section{Cluster braid groups: weighted case}
%====================================================

%====================================================
\subsection{Cluster exchange graphs}\
%====================================================

Let $(Q, W)$ be a quiver with potential that mutation equivalent to $\Lambda$. Denote by $\Gamma(Q,W)$ be the \emph{Ginzburg dg algebra} associated to $(Q,W)$ (c.f. \cite[\S 7.2]{K12} for the definition). Let $\pvd(\Lambda)$ be the \emph{finite-dimensional (or perfectly valued) derived category} of $\Gamma(Q,W)$, and $\per(\Lambda)$ be the \emph{perfect derived category} of $\Gamma(Q,W)$. The Verdier quotient $\C(\Lambda):=\per(\Lambda)/\pvd(\Lambda)$ is known as the \emph{cluster category} of $\Lambda$.

An object $\Y$ in $\C(\Lambda)$ is said to be \emph{rigid} if it has no self-extensions, and is a \emph{cluster tilting object} (CTO for short) if, furthermore, $\Y$ is basic and maximal with respect to this property.

\begin{Def}[{\cite{IY}}]\label{def:IY}
	Let $\Y=\oplus^n_{i=1}\Y_i$ be a basic CTO in $\C(\Lambda)$. The \emph{forward mutation} $\mu^\sharp_i(\Y)$ of $\Y$ at $\Y_i$ is defined as the CTO obtained by replacing $\Y_i$ with $\Y^\sharp_i$, where
	\[\Y^\sharp_i=\Cone(\Y_i\to\oplus_{j\neq i}\Irr(\Y_i,\Y_j)^*\otimes \Y_j)\]
	with $\Irr(\Y_i,\Y_j)$ the space of irreducible morphisms between $\Y_i$ and $\Y_j$.
\end{Def}

The \emph{oriented cluster exchange graph} $\CEG(\Lambda)$ of the cluster category $\C(\Lambda)$ is the graph whose vertices are CTOs and whose edges correspond to forward mutations. The \emph{unoriented cluster exchange graph} $\uCEG(\Lambda)$ of the cluster category $\C(\Lambda)$ is the oriented graph obtained from $\CEG(\Lambda)$ by replacing each 2-cycle given by $\mu^\sharp_i$ with a single edge $\mu_i$.
\begin{lem}[{\cite[Prop.~4.5]{Q15}}]
    $\pi_1(\uCEG(\Lambda))$ is generated by squares and pentagons.
\end{lem}

For each $\oi\in\Delta_0$, the \emph{weighted forward mutation} is defined as $\mu^\sharp_\oi=\prod_{i\in f^{-1}(\oi)}\mu^\sharp_i$, which, as the categorical interpretation of $\mu_\oi(\Delta)$, is well-defined. The \emph{weighted cluster tilting object} (weighted CTO for short) is the object $\Y$ in $\C(\Delta)$  obtained from the initial CTO $\Y_\Delta:=\k\overrightarrow{\Delta}$ via weighted-mutations.

The oriented cluster exchange graph $\CEG(\Delta,f)$ of $\Delta$ has weighted CTOs as vertices and weighted mutations as edges. The unoriented cluster exchange graph $\uCEG(\Delta,f)$ is obtained from $\CEG(\Delta,f)$ by replacing each 2-cycle given by $\mu^\sharp_\oi$ with a single edge $\mu_\oi$.

\begin{remark}
    The construction of generalized associahedron in \cite{FZ1} is valid for the Coxeter-Dynkin diagram of type $H$ and $I$. The 1-skeleton of the associahedron is exactly the cluster exchange graph. In \cite[\S~5.3]{FR}, the 1-skeleton of the associahedron of type $H_3$ was constructed. In \cite{DT,DT2}, there is a categorical approach to construct the cluster exchange graph.
\end{remark}

\begin{lem}[{\cite[Thm.~5.2]{QZx}}]\label{prop:ceg}
The unoriented cluster exchange graph $\uCEG(\Delta,f)$ is a generalized associahedron of $\overrightarrow{\Delta}$ in \cite{FZ1}, i.e. independent of $f$.
\end{lem}
In the rest of this paper, we may use $\uCEG(\Delta)$ (resp. $\CEG(\Delta)$) to denote the unoriented (resp. oriented) cluster exchange graph of $\Delta$.

\begin{Ex}\label{ex:i2-gon}
    A direct calculation shows that $\uCEG(\I)$ is an $(m+2)$-gon as shown by the first picture in Figure~\ref{fig:ceg i2}, whose oriented version is as shown by the right picture.
\end{Ex}
\begin{figure}[h]\centering
	\newcommand{\vsource}{\otimes}
	\newcommand{\vsink}{\odot}
	\newcommand{\vertx}{\bullet}
	\begin{tikzpicture}[scale=.9, rotate=90,
		arrow/.style={->,>=stealth}]		
		\foreach \j in {1,2,3,4,5}
		{\draw (72*\j:2cm) node (t\j) {$\bullet$};}
		\foreach \j in {1,2,3}
		{\draw (72-72*\j:2.4cm) node {$\j$};}
		\draw (72:3cm) node {$m+2$}(72-72*4:2.5cm) node {$m+1$};
		
		\foreach \j in {4,1}
		{\draw (72*\j+36:2.1cm) node {$\mu_\1$};}
		\foreach \j in {3,5}
		{\draw (72*\j+36:2.1cm) node {$\mu_\2$};}		
		
		\foreach \a/\b in {4/3,5/4,5/1,1/2}{
			\draw (t\a) edge[]  (t\b);}
		\draw (t2) edge[dashed] (t3);
	\end{tikzpicture}
	\qquad
	\begin{tikzpicture}[scale=.9, rotate=90,
		arrow/.style={->,>=stealth}]
		\foreach \j in {1,2,3,4,5}
		{\draw (72*\j:2cm) node (t\j) {$\bullet$};}
		\foreach \j in {1,2,3}
		{\draw (72-72*\j:2.4cm) node {$\j$};}
		\draw (72:3cm) node {$m+2$}(72-72*4:2.6cm) node {$m+1$};
		\draw(36:1.6)node{$\2$};
        \draw(-36:1.6)node{$\1$};
        \draw(-108:1.6)node{$\2$};
        \draw(108:1.6)node{$\1$};
		\foreach \j in {3,4,5,1}
		{\draw (72*\j+36:2.1cm) node {$\x$};\draw (72*\j+36:1cm) node {$\y$};}
		\foreach \a/\b in {3/4,4/5,5/1,1/2}{
			\draw (t\a) edge[arrow,bend left=16]  (t\b);
			\draw (t\b) edge[arrow,bend left=16]  (t\a);}
		\draw (t2) edge[->,dashed,bend left=16] (t3);
		\draw (t3) edge[->,dashed,bend left=16] (t2);
	\end{tikzpicture}
	\caption{unoriented and oriented cluster exchange graph of $\I$}
	\label{fig:ceg i2}
\end{figure}
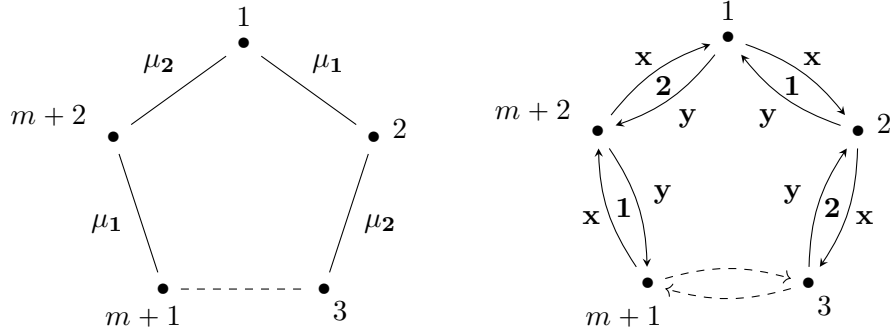

For each weighted CTO $\Y$ in $\CEG(\Delta)$, denote by $(\wtq_\Y,\W_\Y)$ the weighted QP corresponds to $\Y$, which is folded from $(Q_\Y,W_\Y)$ in $\CEG(\Lambda)$. When $\Delta=\I$, we have $\wtq_\Y=\I$ for all $\Y$ in $\CEG(\Delta)$. Denote by $\x$ (resp. $\y$) the weighted forward mutation at the object that corresponds to the sink (resp. source) of $\I$ (cf. the right picture in Figure~\ref{fig:ceg i2}).
The \emph{polygonal/(m+2)-gon relation} is defined to be
\begin{equation}\label{eq:xy}
  \x^2=\y^m.
\end{equation}

\begin{Ex}
    The unoriented cluster exchange graph of $A_3, B_3/C_3, H_3$ is as shown by Figure~\ref{fig:ABH}, where
    \begin{itemize}
        \item $\uCEG(A_3)$, as shown by the first picture, has three squares and six pentagons;
        \item $\uCEG(B_3)$, as shown by the second picture, has four hexagons
         (two in the front and two in the back) and eight squares/pentagons alternatively around them;
        \item $\uCEG(H_3)$ has six heptagons (three in the front and three in the back) and twelve squares/pentagons alternatively around them.
    \end{itemize}
\begin{figure}  \centering
  \centering\makebox[\textwidth][c]{
    \includegraphics[width=16cm]{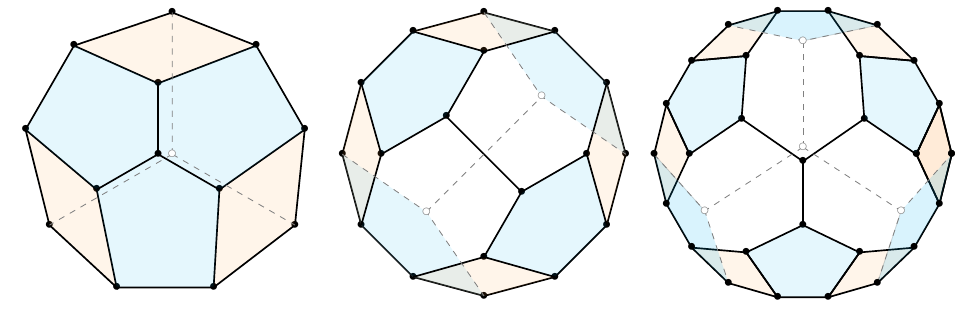}}
\caption{The cluster exchange graphs $\uCEG(\Delta)$ for type $A_3, B_3/C_3$ and $H_3$}
\label{fig:ABH}
\end{figure}
\end{Ex}

\begin{remark}\label{rem:ceg}
For each weighted CTO $\Y$ in $\uCEG(\Delta)$ and any arrow $\epsilon:\oi\xrightarrow{m}\oj$ in $\wtq_\Y$,
there is a mutation sequence of alternatively mutating $\mu^\sharp_\oi$ and $\mu^\sharp_\oj$.
Using Calabi-Yau reduction (cf. \cite[Thm.~4.2]{IY} and \cite[Thm.~7.4]{KV}),
one gets a full subgraph of $\uCEG(\Delta)$ that is isomorphic to $\uCEG(I_2(m))$,
i.e., an $(m+2)$-gon in Example~\ref{ex:i2-gon}.
\end{remark}

We prove an analogue of \cite[Prop.~4.5]{Q15} in the case of Coxeter-Dynkin diagrams
about the fundamental group of $\uCEG$.

\begin{prop}\label{prop:fund gp}
For any finite weighted quiver $\overrightarrow{\Delta}$, $\pi_1(\uCEG(\Delta))$ is generated by $(m+2)$-gons
of the form in \Cref{ex:i2-gon}, for $m\in\{\w(\alpha)|\alpha\in\overrightarrow{\Delta}_1\}$.
\end{prop}
\begin{proof}
The assertion is followed by a similar discussion as in \cite[Prop.~4.5]{Q15}.

More precisely, fixing a finite weighted folding $\Fold$, we have the following.
\begin{itemize}
  \item
Recall that $\uCEG(\Delta)$ can be embedded in $\uCEG(\Lambda)$ in the sense that
the vertex set of $\uCEG(\Delta)$ is a subset of $\uCEG(\Lambda)$ and
edges of $\uCEG(\Delta)$ correspond to paths in $\uCEG(\Lambda)$.
Alternatively, if the finite weighted folding $\Fold$ comes from a fusion action (cf. \cite[\S~3]{QZx}), $\uCEG(\Delta)$ is exactly the fusion-stable part of $\uCEG(\Lambda)$.
  \item
The same assertion holds for the silting exchange graph.
Namely, if we consider simultaneously multi-mutation in \cite{DT} (or fusion-stable mutation in \cite{QZx}), one can introduce and embed the silting exchange graph $\SEG(\Delta)$ into $\SEG(\Lambda)$.
  \item
The cluster exchange graph $\uCEG(\Lambda)$ can be identified with the (underlying graph of) the interval $[\k\overrightarrow{\Lambda},\k\overrightarrow{\Lambda}[1]]$ of the silting exchange graph $\SEG(\Lambda)$, cf. \cite[Thm.~5.14]{KQ}.
The same statement holds for $\uCEG(\Delta)$. Thus, it is enough to show that the fundamental group of $\SEG(\Delta)$ is generated by oriented $(m+2)$-gons.
  \item Similar to \Cref{rem:ceg}, there are indeed oriented $(m+2)$-gons
  in $\SEG(\Delta)$ via silting reduction (cf. \cite[Lem.~4.3]{Q15} or \cite[Lem.~6.1]{QW}
  for the dual statement of exchange graphs of hearts).
  Note that any such oriented $(m+2)$-gon is of the form \eqref{eq:xy},
  in the sense that there exists a source and a sink of the $(m+2)$-gon
  and the two paths from the source to the sink have length 2 and $m$ respectively.
  Finally, one can follow the proof strategy of \cite[Prop.~4.5]{Q15} to show that $\pi_1(\SEG(\Delta))$ is generated by oriented $(m+2)$-gons.
\end{itemize}
\end{proof}

%===========================================================
\subsection{Cluster braid groups for weighted QPs}\
%===========================================================

Given a weighted graph $\Delta$, denote by $\mathcal{W}(\CEG(\Delta))$ the \emph{path groupoid} of $\CEG(\Delta)$, i.e., the category whose objects are vertices in $\CEG(\Delta)$ and whose generating morphisms are edges of $\CEG(\Delta)$ and their formal inverse. Recall that $\x$ (resp. $\y$) is the weighted forward mutation at the object that corresponds to the sink (resp. source) of a $\I$ subgraph of $\Delta$.
There is a polygonal/(m+2)-gon relation of the form $ \x^2=\y^m$ \eqref{eq:xy}.
By \Cref{prop:fund gp}, fundamental group $\pi_1(\uCEG(\Delta))$ is generated by (m+2)-gon relations. We could construct a groupoid from $\CEG(\Delta)$ by adding these (m+2)-gon relations.

\begin{Def}\label{def:ceg-wtq}
The \emph{oriented cluster exchange groupoid} $\ceg(\Delta)$ is defined to be the quotient of $\mathcal{W}(\CEG(\Delta))$ by all polygonal relations.

For each weighted CTO $\Y$ and $\oi\in\wtq_{\Y}$,
the \emph{local twist} $t^\Y_\oi$ is the 2-cycle in $\mathcal{W}(\CEG(\Delta))$ corresponding to the weighted forward mutation $\mu^\sharp_{\oi}$ at $\Y$.

The \emph{cluster braid group} $\CBr_\Delta(\Y)$ is defined to be the subgroup of $\pi_1(\ceg(\Delta),\Y)$ generated by all local twists $t^\Y_\oi,\oi\in\Delta_0$.
\end{Def}

Let $\Y$ be a weighted CTO in $\CEG(\Delta)$ and $f^{-1}(\oi)=\{i_1,\cdots,i_p\}$. Denote by $\x$ (resp. $\y$) the weighted forward mutation of $\Y$ (resp. $\Y'=\mu_\oi(\Y)$) at $\oi$.
Denote by $x_{i_j}$ (resp. $y_{i_j}$), $1\leq j\leq p$ the forward mutation of the corresponding object at $i_j$, see Figure~\ref{fig:local twist}.
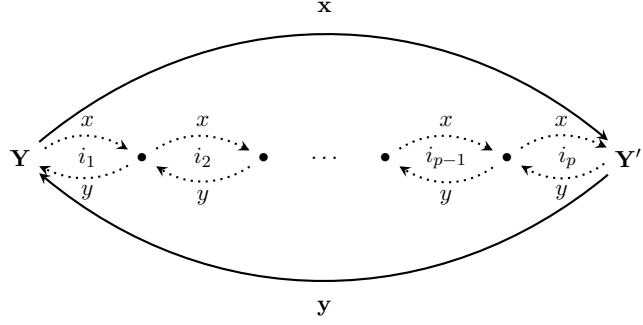
\begin{figure}[htpb]\centering
    \begin{tikzpicture}[scale=0.8,transform shape]
        \node (a)at (-5,0){$\Y$};
        \node (b)at (-3,0){$\bullet$};
        \node (c)at (-1,0){$\bullet$};
        \node (d)at (1,0){$\bullet$};
        \node (e)at (3,0){$\bullet$};
        \node (f)at (5,0){$\Y'$};
        \draw(0,0)node{$\ldots$};

        \draw[thick,-{stealth}](a)to[out=40,in=140](f);
        \draw[thick,-{stealth}](f)to[out=220,in=-40](a);

        \draw[thick,-{stealth},dotted](-4.6,.15)to[out=35,in=150](b);
        \draw[thick,-{stealth},dotted](b)to[out=32,in=153](c);
        \draw[thick,-{stealth},dotted](d)to[out=32,in=153](e);
        \draw[thick,-{stealth},dotted](e)to[out=34,in=155](f);

        \draw[thick,-{stealth},dotted](4.6,-.1)to[out=217,in=-30](e);
        \draw[thick,-{stealth},dotted](e)to[out=215,in=-32](d);
        \draw[thick,-{stealth},dotted](c)to[out=213,in=-32](b);
        \draw[thick,-{stealth},dotted](b)to[out=210,in=-25](a);

        \draw(0,2.5)node{$\x$}(0,-2.5)node{$\y$}(-3.9,0)node{$i_1$}(-2,.6)node{$x$}(-2,-.65)node{$y$}(-2,0)node{$i_2$}(2,.6)node{$x$}(2,-.65)node{$y$}(2,0)node{$i_{p-1}$}(4,0)node{$i_p$}(-3.9,.6)node{$x$}(-3.9,-.58)node{$y$}(3.9,.6)node{$x$}(3.9,-.58)node{$y$};
    \end{tikzpicture}
\caption{Decomposing of local twist $t^\Y_\oi$ in $\CEG(\Delta)$ as local twists (dotted 2-cycles) in $\CEG(\Lambda)$}
\label{fig:local twist}
\end{figure}

By \cite[Prop.~2.8]{KQ2} and commutativity between mutations, we have
\begin{equation}\label{eq:local twist}
	t^\Y_\oi=\x\y=(x_{i_1}\cdots x_{i_p})(y_{i_p}\cdots y_{i_1})=t^\Y_{i_p}\cdots t^\Y_{i_1}=\prod_{k\in f^{-1}(\oi)}t^\Y_k.
\end{equation}

Similarly to \cite[\S~2.4]{KQ2}, we define the conjugation map
\begin{gather*}
    \begin{array}{rcl}
        \ad_\oi:\pi_1(\ceg(\Lambda),\Y) & \longrightarrow & (\ceg(\Delta),\Y')\\
        t & \mapsto & \x^{-1} t\x.
    \end{array}
\end{gather*}
Note that
\begin{equation}\label{eq:ad}
	\ad_\oi(t)=\x^{-1}t\x=(x^{-1}_{i_p}\cdots x^{-1}_{i_2}x^{-1}_{i_1})t(x_{i_1}x_{i_2}\cdots x_{i_p})=\ad_{i_p}\circ\cdots\circ\ad_{i_2}\circ\ad_{i_1}(t).
\end{equation}
Under the above setting, we have the analogue of \cite[Prop.~2.9]{KQ2}.

\begin{prop}\label{lem:conj}
	Let \begin{tikzcd}
	    \Y\ar[r,bend left=17,"\x","\oi"'] & \Y'\ar[l,bend left=17,"\y"]
	\end{tikzcd}
    be a subgraph in $\CEG(\Delta)$ with $\x$ (resp. $\y$) the weighted forward mutation of $\Y$ (resp, $\Y'$) at $\oi\in\Delta_0$. Then
	\[\ad_\oi(t_\oj)=\begin{cases}
		(t'_\oi)^{-1}t'_\oj t'_\oi&\text{ if there are arrows from $\oj$ to $\oi$ in $\overrightarrow{\Delta}_1$,}\\
		t'_\oj&\text{otherwise,}
	\end{cases}\]
	where $t_\oj$ (resp. $t'_\oj$), $\oj\in\Delta_0$, is the local twist based at $\Y$ (resp. $\Y'$). Hence $\ad_\oi$ induces an isomorphism between $\CBr_\Delta(\Y)$ and $\CBr_\Delta(\Y')$. In particular, we have $\pi_1(\ceg(\Delta),\Y)\cong\CBr_\Delta(\Y)$.
\end{prop}
\begin{proof}
	If there is no arrow in $\overrightarrow{\Delta}$ from $\oj$ to $\oi$, then there is no arrows in $\overrightarrow{\Lambda}$ from $f^{-1}(\oj)$ to $f^{-1}(\oi)$. Let $f^{-1}(\oi)=\{i_1,i_2,\cdots,i_p\}$ and $f^{-1}(\oj)=\{j_1,j_2,\cdots,j_q\}$. By equation~\eqref{eq:local twist}, ~\eqref{eq:ad} and \cite[Prop.~2.8]{KQ2}, we have
	\[\ad_\oi(t_\oj)=\ad_{i_p}\circ\cdots\circ\ad_{i_2}\circ\ad_{i_1}(t_{j_q}\cdots t_{j_2}t_{j_1})=(t'_{j_q}\cdots t'_{j_2}t'_{j_1})=t'_\oj.\]
	Similarly, if there is an arrow from $\oj$ to $\oi$ then we have
	\[\begin{aligned}
		\ad_\oi(t_\oj){}&=\ad_\oi(t_{j_q})\cdots\ad_\oi(t_{j_2})\ad_\oi(t_{j_1}){}\\
        &=(\ad_{i_1}\circ\cdots\circ\ad_{i_p})(t_{j_q})\cdots(\ad_{i_1}\circ\cdots\circ\ad_{i_p})(t_{j_1}){}\\
		&=\big((t_{i_p}\cdots t_{i_1})^{-1}t_{j_q}(t_{i_p}\cdots t_{i_1})\big)\cdots\big((t_{i_p}\cdots t_{i_1})^{-1}t_{j_1}(t_{i_p}\cdots t_{i_1})\big){}\\
        &=(t_{i_p}\cdots t_{i_1})^{-1}(t_{j_q}\cdots t_{j_1})(t_{i_p}\cdots t_{i_1})=(t'_\oi)^{-1}t'_\oj t'_\oi
    \end{aligned}\]
	where the third equality follows from \cite[Lem.~2.7]{KQ2}.
	
	By Proposition~\ref{prop:fund gp}, $\pi_1(\uceg(\Delta),\Y)$ is trivial. A similar argument to \cite[Prop.~2.9]{KQ2} shows that $\pi_1(\ceg(\Delta),\Y)\cong\CBr_\Delta(\Y)$. So the last assertion follows.
\end{proof}

%=========================================================
\subsection{Cluster braid groups as braid groups}\
%=========================================================
We simply denote $\CBr(\Delta)$ the cluster braid group $\CBr_\Delta(\Y_\Delta)$.

\begin{lem}\label{lem:surj}
	The following map is a group isomorphism
    \begin{gather}
        \begin{array}{rcl}
            \Br\I&\longrightarrow&\CBr(\I) \\
            b_\oi&\mapsto&t_\oi,
        \end{array}
    \end{gather}
	where $b_\oi$ (resp. $t_\oi$) is the generator of $\Br\I$ (resp. $\CBr(\I)$) corresponds to the vertex $\oi$ of $\I$.
\end{lem}
\begin{proof}
Recall that $\Br\I$ has a canonical presentation
\begin{equation}\label{eq:br}
    \langle b_\1,b_\2 \mid \Br^m(b_\1,b_\2)
\rangle.
\end{equation}
There is an alternative presentation of $\Br\I$:
\begin{itemize}
    \item If $2\nmid m$, then the presentation \ref{eq:br} is equivalent to
    \[\langle u,v\mid u^2=v^m\rangle\] for
    \[\begin{cases}
    u=\underbrace{b_\1b_\2b_\1\ldots b_\1}_m\\
  v=b_\1b_\2
\end{cases}.\]
\item If $2\mid m$, then the presentation \ref{eq:br} is equivalent to
\[\langle \sigma,\tau\mid (\sigma\tau)^{m/2}=(\tau\sigma)^{m/2}\rangle\]
for
\[\begin{cases}
    \sigma=b_\1 b_\2 b_\1\\
  \tau=b_\1^{-1}
\end{cases}.\]
\end{itemize}

Recall that the groupoid $\ceg(\I)$ is constructed from $(m+2)$-gon by replacing each edge by a 2-cycle and adding the polygon relation $x^2=y^m$. By \Cref{lem:conj}, the group $\CBr(\I)$ is isomorphic to the fundamental group of $\ceg(\I)$. Moreover, local twists $t_{\1}$ and $t_{\2}$ are generators of $\pi_1(\ceg(\I))$.

{\bf Case 1.} If $2\nmid m$, we take two new generators 
\[u=x_1\ldots x_{m+1}x_{m+2}=\underbrace{t_{\1}t_{\2}t_{\1}\ldots t_\1}_m\text{ and } v=y_{m+2}\ldots y_2y_1=t_{\1}t_{\2}.\]
By a direct calculation, $u^2=v^m$. Moreover, this is the only relation for $\pi_1(\ceg(\I))$, since $u^2$ and $v^m$ form the boundary of the only 2-cell in $\ceg(\I)$, see Figure~\ref{fig:modd}. Therefore $\Br\I=\CBr(\I)$.

{\bf Case 2.} If $2\mid m$, we take two generators $\sigma=t_{\1}v$ and $\tau=t_{\1}^{-1}$, where $v=y_{m+2}\ldots y_2y_1=t_{\2}t_{\1}$ and $t_{\1}=x_1y_1$. Since $t_{\2}=\tau\sigma\tau$ and $t_{\1}=\tau^{-1}$. A direct calculation shows that $(\sigma\tau)^{m/2}=(\tau\sigma)^{m/2}$. Moreover, this is the unique relation for $\pi_1(\ceg(\I))$. For $m=4$, see Figure~\ref{fig:meven}. Thus $\Br\I=\CBr(\I)$ in this case.

It follows from the definition that $t_{\1}=xy, t_\2=yx$ and $x^2=y^m$, see the right picture in Figure~\ref{fig:ceg i2}. The following two equations show that $t_{\1}$ and $t_{\2}$ satisfy the braid relation.
	\begin{gather}
		\underbrace{t_{\1}t_\2 t_{\1}\cdots}_m=\begin{cases}
			y^{m(l+1)}, & m=2l,\\
			xy^{m(l+1)}, & m=2l+1;
		\end{cases}
	\end{gather}
	and
	\begin{gather}
		\underbrace{t_\2 t_{\1}t_\2\cdots}_m=\begin{cases}
			y^{m(l+1)}, & m=2l,\\
			y^{m(l+1)}x, & m=2l+1.
		\end{cases}
	\end{gather}
Thus $t_{\1}=xy, t_\2=yx$ are the generators of $\CBr(\I)$ which satisfy the braid relation. $b_\oi\mapsto t_\oi, \oi\in\{\1,\2\}$ gives an isomorphism between $\Br\I$ and $\CBr(\I)$
\end{proof}
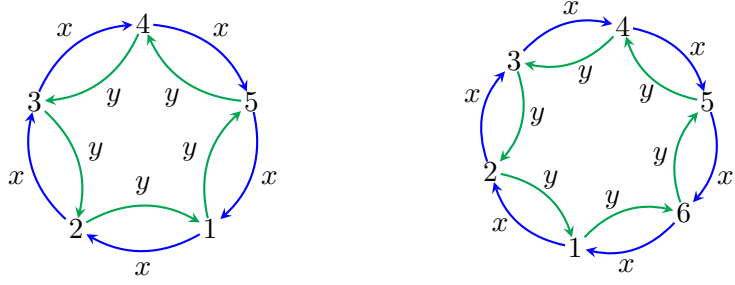
\begin{figure}
	\centering
	\begin{tikzpicture}[Rline/.style={thick,red}, Bline/.style={thick,blue},Gline/.style={thick,dotted,green},
		bline/.style={thick,dotted,blue},rline/.style={thick,dotted,red},>=stealth,scale=0.5]
		
        \begin{scope}[shift={(-6,1)}]
			\foreach \j in {0,1,2,3,4}
			{\node (a\j) at (18+72*\j:3){};
				\draw (-18+72*\j:3.5)node {$x$};
				\draw (-18+72*\j:1.3)node {$y$};}
			
			\draw (a0)node {5};
			\draw (a1)node {4};
			\draw (a2)node {3};
			\draw (a3)node {2};
			\draw (a4)node {1};
			
			\foreach \a/\b in {0/1,1/2,2/3,3/4,4/0}
			\draw [-{stealth},thick,Green] (a\a) [bend left] to (a\b);
			
			\foreach \a/\b in {1/0,0/4,4/3,3/2,2/1}
			\draw [-{stealth},thick,blue] (a\a) [bend left] to (a\b);
		\end{scope}

		\begin{scope}[shift={(6,1)}]
			\foreach \j in {0,1,2,3,4,5}
			{\node (a\j) at (18+60*\j:3){};
				\draw (-18+60*\j:3.5)node {$x$};
				\draw (-18+60*\j:1.7)node {$y$};}
			
			\draw (a0)node {5};
			\draw (a1)node {4};
			\draw (a2)node {3};
			\draw (a3)node {2};
			\draw (a4)node {1};
			\draw (a5)node {6};
			
			\foreach \a/\b in {0/1,1/2,2/3,3/4,4/5,5/0}
			\draw [-{stealth},thick,Green] (a\a) [bend left] to (a\b);
			
			\foreach \a/\b in {1/0,0/5,5/4,4/3,3/2,2/1}
			\draw [-{stealth},thick,blue] (a\a) [bend left] to (a\b);
		\end{scope}
		
	\end{tikzpicture}
	\caption{$\CEG(\Delta)$ for $\Delta$ of types $I_2(3)$ and $I_2(4)$}\label{fig:i_2,4}
\end{figure}

\begin{figure}
    \centering
\begin{tikzpicture}[scale=1.3,yscale=1.2]
\draw[cyan!10,fill=cyan!10] (0,0)to(5,1)to(10,0);
\draw[dashed,red,thick,->-=.4,>=stealth] (0,0)to(5,1);
\draw[dashed,red,thick,->-=.4,>=stealth] (5,1)to(10,0);
\draw[blue,thick,->-=.5,>=stealth](10,0)to(11,1);
\foreach \x in {1,2,3,4,5}{
\begin{scope}[shift={(2*\x-2,0)}]
  \foreach \a/\b/\c/\d in {0/0/1/1,1/1/2/0} {
  \draw[blue,thick,->-=.5,>=stealth] (\a,\b) to (\c,\d);}
  \foreach \a/\b/\c/\d in {0/0/0.5857/-0.5857,0.5857/-0.5857/2-0.5857/-0.5857,2-0.5857/-0.5857/2/0} {
  \draw[cyan!10,fill=cyan!10] (\a,\b) to (\c,\d) to (1,0) --cycle;
  \draw[red](0,0)\nn node[left]{\footnotesize{$V_\x$}};
  \draw[Green,thick,->-=.5,>=stealth] (\a,\b) to (\c,\d);}
\draw[orange](1,0)node{$\x$};
\end{scope}
}
\foreach \x in {4,5}{
\begin{scope}[shift={(2*\x-7,1)}]
  \foreach \a/\b/\c/\d in {0/0/0.5857/0.5857,0.5857/0.5857/2-0.5857/0.5857,2-0.5857/0.5857/2/0} {
  \draw[Green,thick,->-=.5,>=stealth] (\a,\b) to (\c,\d);}
\draw[orange](1,0)node{$\x$};  \draw[red](0,0)\nn node[left]{\footnotesize{$V_\x$}};
\end{scope}
}
\foreach \x in {1,2,3}{
\begin{scope}[shift={(2*\x+3,1)}]
  \foreach \a/\b/\c/\d in {0/0/0.5857/0.5857,0.5857/0.5857/2-0.5857/0.5857,2-0.5857/0.5857/2/0} {
  \draw[Green,thick,->-=.5,>=stealth] (\a,\b) to (\c,\d);}
    \draw[red](0,0)\nn node[left]{\footnotesize{$V_\x$}};
\draw[orange](1,0)node{$\x$};
\end{scope}
}
\draw[blue]
    (.4,.6)node{\footnotesize{$x_1$}}
    (1.6,.6)node{\footnotesize{$x_2$}}
    (2+.4,.6)node{\footnotesize{$x_3$}}
    (2+1.6,.6)node{\footnotesize{$x_4$}}
    (2+2+.4,.6)node{\footnotesize{$x_5$}}
    (2+2+1.6,.6)node{\footnotesize{$x_1$}}
    (6+.4,.6)node{\footnotesize{$x_2$}}
    (6+1.6,.6)node{\footnotesize{$x_3$}}
    (8+.4,.6)node{\footnotesize{$x_4$}}
    (8+1.6,.6)node{\footnotesize{$x_5$}};
\draw[red]
    (0,0)\nn node[left]{\footnotesize{$V_1$}}    (10,0)\nn node[right]{\footnotesize{$V_1$}}
        (4-0.5857,-0.5857)\nn node[below]{\footnotesize{$V_1$}}
        (6+0.5857,-0.5857)\nn node[below]{\footnotesize{$V_1$}};
\draw[Green]
    (.2,-.3555555555)node{\footnotesize{$y_5$}}
    (1,-.7)node{\footnotesize{$y_4$}}
    (1.8,-.35)node{\footnotesize{$y_3$}}
    (2+.2,-.3555555555)node{\footnotesize{$y_2$}}
    (2+1,-.7)node{\footnotesize{$y_1$}}
    (2+1.8,-.35)node{\footnotesize{$y_5$}}
    (4+.2,-.3555555555)node{\footnotesize{$y_4$}}
    (4+1,-.7)node{\footnotesize{$y_3$}}
    (4+1.8,-.35)node{\footnotesize{$y_2$}}
    (6+.2,-.3555555555)node{\footnotesize{$y_1$}}
    (6+1,-.7)node{\footnotesize{$y_5$}}
    (6+1.8,-.35)node{\footnotesize{$y_4$}}
    (8+.2,-.3555555555)node{\footnotesize{$y_3$}}
    (8+1,-.7)node{\footnotesize{$y_2$}}
    (8+1.8,-.35)node{\footnotesize{$y_1$}};
\end{tikzpicture}
    \caption{Unpacked view of $I_2(3)$}
    \label{fig:modd}
\end{figure}
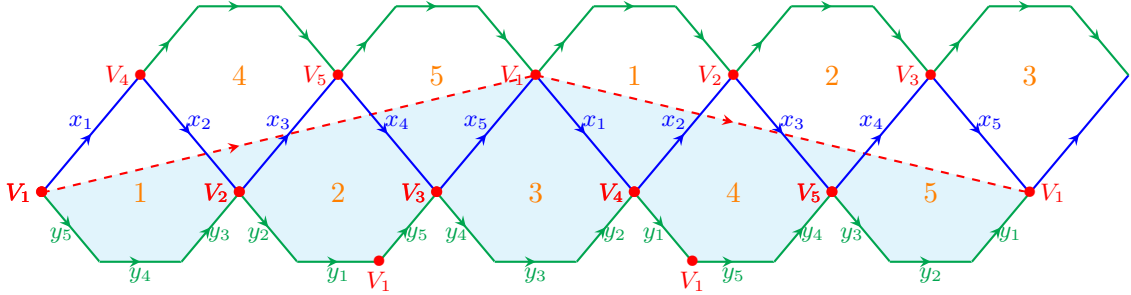

\begin{figure}
	\centering
	% Requires \usepackage{graphicx}
	\begin{tikzpicture}[scale=0.7,transform shape]
		
		\filldraw[cyan!15,thick,dashed] (-12,0) .. controls (-8,2) and (-4,2) ..
		(0,0) --(-1,-2)--(-3,-2)--(-4,0)--(-5,-2)--(-7,-2)--(-8,0)--(-9,-2)--(-11,-2);
		
		\filldraw[cyan!15,thick,dashed] (-12,0) .. controls (-8,2) and (-4,2) ..
		(0,0) --(2,2)--(1,4)--(-1,4)--(-2,2)--(-3,4)--(-5,4)--(-6,2)--(-7,4)--(-9,4)--(-10,2);
		
		\foreach \j in {0,1,2,3,4}
		{\node (a\j)[red] at (-12+4*\j,0){$\bullet$};}
		
		\foreach \j in {0,1,2,3,4}
		{\node (b\j)[red] at (-10+4*\j,2){$\bullet$};}
		
		\foreach \j in {0,1,2,4,5,6,8,9,10,12,13,14}
		{\node (c\j)[cyan] at (-11+\j,-2){$\bullet$};}
		%%%%%%%%%%%%%label the polygon%%%%5
		\foreach \j in {1,3,5}
		{\node (\j)[red] at (-12+2*\j,-0.4){\j};}
		\foreach \j in {2,4,6}
		{\node (\j)[red] at (-12+2*\j,2.3){\j};}
		\node [red] at (2,-0.4){1};
		\node [red] at (4,2.3){2};
		%%%%%%%%%%%%%%%%%%%%%%%%
		
		\draw[-{stealth},thick,blue] (a0)to node[above]{$x_1$}  (b0);
		\draw[-{stealth},thick,blue] (a1)to node[above]{$x_3$}  (b1);
		\draw[-{stealth},thick,blue] (a2)to node[above]{$x_5$}  (b2);
		\draw[-{stealth},thick,blue] (a3)to node[above]{$x_1$}  (b3);
		\draw[-{stealth},thick,blue] (a4)to node[above]{$x_1$}  (b4);
		
		\draw [-{stealth},thick,blue] (b0)  to node[above,blue]{$x_2$} (a1);
		\draw [-{stealth},thick,blue] (b1)  to node[above,blue]{$x_4$} (a2);
		\draw [-{stealth},thick,blue] (b2)  to node[above,blue]{$x_6$} (a3);
		\draw [-{stealth},thick,blue] (b3)  to node[above,blue]{$x_2$} (a4);

		%%%%%%%%%%%%%%%%%%%%%%%%%%%

		\node[green!70!black] at (-11.8,-0.8){$y_6$};
		\node[green!70!black] at (-7.8,-0.8){$y_2$};
		\node[green!70!black] at (-3.8,-0.8){$y_4$};
		\node[green!70!black] at (0.2,-0.8){$y_1$};
		
		\node[green!70!black] at (-8.7,-0.8){$y_3$};
		\node[green!70!black] at (-4.7,-0.8){$y_5$};
		\node[green!70!black] at (-0.7,-0.8){$y_1$};
		%%%%%%%%%%%%%%%%%%%%%%%%%%%%%
		\node[green!70!black,below] at (-10.5,-2){$y_5$};
		\node[green!70!black,below] at (-9.5,-2){$y_4$};
		\node[green!70!black,below] at (-6.5,-2){$y_1$};
		\node[green!70!black,below] at (-5.5,-2){$y_6$};
		\node[green!70!black,below] at (-2.5,-2){$y_3$};
		\node[green!70!black,below] at (-1.5,-2){$y_2$};
		%%%%%%%%%label the top row y_i%%%%%%%%%%%%%%%%%%%%
		\node[green!70!black,below] at (-9.7,3.4){$y_1$};
		\node[green!70!black,below] at (-5.7,3.4){$y_3$};
		\node[green!70!black,below] at (-1.7,3.4){$y_5$};
		\node[green!70!black,below] at (2.3,3.4){$y_1$};
		
		\node[green!70!black,below] at (-6.3,3.4){$y_4$};
		\node[green!70!black,below] at (-2.3,3.4){$y_6$};
		\node[green!70!black,below] at (1.7,3.4){$y_2$};

		\node[green!70!black,above] at (-8.5,4){$y_6$};
		\node[green!70!black,above] at (-7.5,4){$y_5$};
		\node[green!70!black,above] at (-4.5,4){$y_2$};
		\node[green!70!black,above] at (-3.5,4){$y_3$};
		\node[green!70!black,above] at (-0.5,4){$y_4$};
		\node[green!70!black,above] at (0.5,4){$y_3$};
		%%%%%%%%%%%%%%%%%%%%%%%
		
		\foreach \i/\j in {0/0,1/4,2/8,3/12}
		{\draw[-{stealth},thick,green!70!black](a\i)to (c\j);}
		%\foreach \j\in {0,1,2,3}
		%{\draw[green!70!black](-11+\j,-2) node {$y$};}
		
		\foreach \i/\j in {2/1,6/2,10/3,14/4}
		{\draw[-{stealth},thick,green!70!black](c\i)to (a\j);}
		
		\foreach \i/\j in {0/1,1/2,4/5,5/6,8/9,9/10,12/13,13/14}
		{\draw[-{stealth},thick,green!70!black](c\i)to (c\j);}	
		
		\foreach \j in {0,1,2,4,5,6,8,9,10,12,13,14}
		{\node (d\j)[cyan] at (-9+\j,4){$\bullet$};}
		
		\foreach \i/\j in {0/0,1/4,2/8,3/12}
		{\draw[-{stealth},thick,green!70!black](b\i)to(d\j);}
		
		\foreach \i/\j in {2/1,6/2,10/3,14/4}
		{\draw[-{stealth},thick,green!70!black](d\i)to(b\j);}
		
		\foreach \i/\j in {0/1,1/2,4/5,5/6,8/9,9/10,12/13,13/14}
		{\draw[-{stealth},thick,green!70!black](d\i)to(d\j);}

		%%%%%%%%%%%%%%%%%%%%%%%%%%%
		\draw[red,thick,dashed,-{stealth}] (-12,0) .. controls (-8,2) and (-4,2) ..(0,0);

		%%%%%%%%%%%%%%%
		\draw[red](-12,0) node[above]{\normalsize{$V_1$}};
		\draw[red](-8,0) node[above]{\normalsize{$V_3$}};
		\draw[red](-4,0) node[above]{\normalsize{$V_5$}};
		\draw[red](4,0) node[above]{\normalsize{$V_3$}};
		\draw[red](0,0) node[above]{\normalsize{$V_1$}};
		
		\draw[red](-9,4) node[above]{\normalsize{$V_1$}};
		\draw[red](-8,4) node[above]{\normalsize{$V_6$}};
		\draw[red](-7,4) node[above]{\normalsize{$V_5$}};
		\draw[red](-5,4) node[above]{\normalsize{$V_3$}};
		\draw[red](-4,4) node[above]{\normalsize{$V_2$}};
		\draw[red](-3,4) node[above]{\normalsize{$V_1$}};
		\draw[red](-1,4) node[above]{\normalsize{$V_5$}};
		\draw[red](0,4) node[above]{\normalsize{$V_4$}};
		\draw[red](1,4) node[above]{\normalsize{$V_3$}};
		
		\draw[red](-10,2) node[left]{\normalsize{$V_2$}};
		\draw[red](-6,2) node[left]{\normalsize{$V_4$}};
		\draw[red](-2,2) node[left]{\normalsize{$V_6$}};
		\draw[red](2,2) node[left]{\normalsize{$V_2$}};
		\draw[red](6,2) node[left]{\normalsize{$V_4$}};

		%%%%%%%%%%%%%%%%%%%%%%%%%
		\draw[red](-11,-2) node[below]{\normalsize{$V_6$}};
		\draw[red](-10,-2) node[below]{\normalsize{$V_5$}};
		\draw[red](-9,-2) node[below]{\normalsize{$V_4$}};
		
		\draw[red](-7,-2) node[below]{\normalsize{$V_2$}};
		\draw[red](-6,-2) node[below]{\normalsize{$V_1$}};
		\draw[red](-5,-2) node[below]{\normalsize{$V_6$}};
		
		\draw[red](-3,-2) node[below]{\normalsize{$V_4$}};
		\draw[red](-2,-2) node[below]{\normalsize{$V_3$}};
		\draw[red](-1,-2) node[below]{\normalsize{$V_2$}};
		
		%%%%%%%%%%%%%%%%%%%%%%%%%%%%%%%%%%%
	\end{tikzpicture}
	\caption{Unpacked view of $I_2(4)$}\label{fig:meven}
\end{figure}
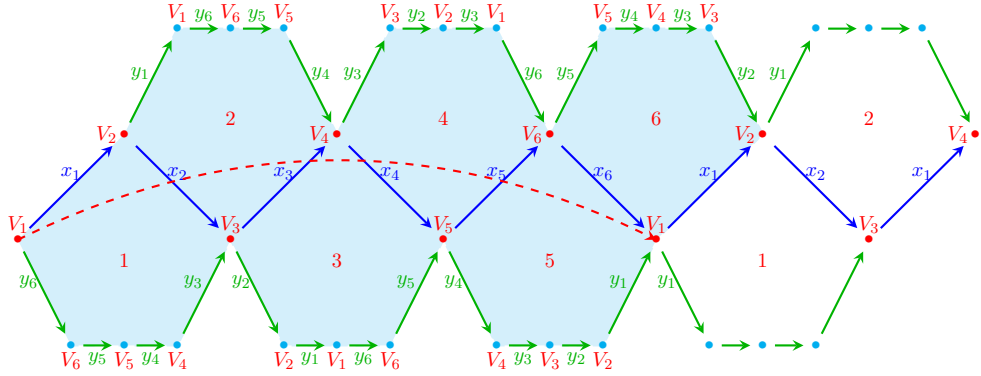

The proof of Proposition~\ref{prop:fund gp} shows that each edge $\epsilon:\oi\overset{m}{-}\oj$ in $\Delta$ gives a doubled $(m+2)$-gon at $\Y_\Delta$ with an $(m+2)$-gon relation. By Lemma~\ref{lem:surj}, the relation $\Br^m(t_\oi,t_\oj)$ holds in $\CBr(\Delta)$, and hence we have a group surjection
\begin{gather}
    \begin{array}{rcl}
        \Psi:\Br\Delta&\longrightarrow&\CBr(\Delta) \\
        b_\oi&\mapsto&t_\oi,
    \end{array}
\end{gather}
where $b_\oi$ (resp. $t_\oi$) is the generator of $\Br\Delta$ (resp. $\CBr(\Delta)$) corresponding to the vertex $\oi\in\Delta_0$.	

Now we state our main theorem which shows that the morphism above is indeed an isomorphism.
\begin{thm}\label{thm:main}
Let $\overrightarrow{\Delta}$ be a finite weighted quiver. For any weighted CTO $\Y$ in $\CEG(\Delta)$, the morphism
\begin{gather}\label{eq:psi}
\begin{array}{rcl}
    \Psi:\Br(\wtq_\Y,\W_\Y)&\longrightarrow&\CBr_\Delta(\Y)\\
        b_\oi&\mapsto& t_\oi
\end{array}
\end{gather}
is a group isomorphism, sending the standard generators to the standard ones (for $\oi\in \wtq_{\Y}$).
\end{thm}
\begin{proof}
By Proposition~\ref{lem:conj} and Theorem~\ref{thm:mu-present}, it is enough to show the isomorphism for the initial weighted CTO $\Y_\Delta$.
By the previous discussion,  we only need to show that $\Psi$ is injective.

Recall that there is a natural embedding from $\CEG(\Delta)$ to $\CEG(\Lambda)$.
Now, we would like to make it more precise, that there is an injective map/functor
\begin{gather}\label{eq:f*}
    f^*\colon\ceg(\Delta)\to\ceg(\Lambda)
\end{gather}
sending $\mu^\sharp_\oi$ to the path $\prod_{i\in f^{-1}(\oi)}\mu^\sharp_i$ (to up homotopy).
Note that on the level of vertices/objects and arrows/morphisms,
$f^*$ is an injective map. The only thing we need to check is that the $(m+2)$-gon relations in $\ceg(\Delta)$ hold in $\ceg(\Lambda)$.
We use a similar method as in the proof of \Cref{prop:fund gp}.

Choose any vertex $\Y$ in $\ceg(\Delta)$ and edge $\oj\xrightarrow{m}\oi\in\wtq_{\Y}$, we have the corresponding $(m+2)$-gon relation.
We need to show that it is generated by oriented square and pentagon relations in $\CEG(\Lambda)$.
By \cite[Thm.~2.10]{KQ2}, $\CEG(\Lambda)=\EG\D_3(\Lambda)/\Br\Lambda$, where $\D_3(\Lambda)$ is the 3-Calabi-Yau category associated to $\overrightarrow{\Lambda}$.
A fundamental domain of $\EG\D_3(\Lambda)/\Br\Lambda$ is the interval $[\h_{\Lambda},\h_{\Lambda}[1]]$ for the canonical heart $\h_{\Lambda}$ of $\D_3(\Lambda)$.
Thus, $\Y$ corresponds to a unique heart $\h_{\Y}$ in $[\h_{\Lambda},\h_{\Lambda}[1]]$.
Moreover, by \cite[\S~10]{KQ} (cf. \cite[Thm.~4.7]{QZx}),
the interval $[\h_{\Y},\h_{\Y}[1]]$ can be regarded as a fundamental domain of $\EG\D_3(\Lambda)/\Br\Lambda$.
Thus the oriented $(m+2)$-gon in $\CEG(\Delta)$ lifts to a cycle in $[\h_{\Y},\h_{\Y}[1]]$.
By \cite[Lem.~6.1]{QW}, such cycle decomposes into squares and pentagons as required.

In particular, \eqref{eq:f*} induces a map, still denoted by $f^*$,
between the corresponding fundamental groups, which fits into the following diagram:
\begin{equation}\label{eq:last}
    \begin{tikzcd}
        \Br\Delta\ar[r,"\Psi_\Delta"]\ar[d,"\iota_f"',hookrightarrow] & \CBr(\Delta)\ar[d,"f^\ast"]\\
        \Br\Lambda\ar[r,"\Psi_\Lambda"']&\CBr(\Lambda),
    \end{tikzcd}
\end{equation}
where $\Psi_\Lambda$ is the isomorphism in \cite[Thm~2.16]{KQ2} and $\iota_f$ is the group injection in Theorem~\ref{thm:cri}.
Examining the standard generators, one gets that such a diagram is commutative.
Thus, $\Psi_\Delta$ is injective because $\Psi_\Lambda\circ\iota_f$ is,
which completes the proof.
\end{proof}

\begin{Ex}
    Consider the weighted graph $\Delta=H_3$, we have the cluster exchange graph $\uCEG(H_3)$ in Figure \ref{fig:ABH}. Denote $\mathbf{Y}$ the common vertex of three heptagons in the front of $\uCEG(H_3)$. The corresponding quiver $\wtq_\Y$ is the weighted quiver IV in Figure \ref{fig:cycle-present}. The potential $\W_\Y$ is the 3-cycle. Then we have
    \begin{equation}
        \CBr_{H_3}(\mathbf Y)\cong\Br(\wtq_\Y,\W_\Y)
    \end{equation}
    which is generated by $b_\1,b_\2,b_\3$ and the  following relations
    \[\Br^5(b_\1,b_\2),\quad \Br^5(b_\2,b_\3), \quad \Br^5(b_\1,b_\3),\quad \Co(b_\1,b^{b_\3b_\2}_\2),\quad  \Br(b_\1,b^{b_\2}_\3).\]
\end{Ex}

\begin{remark}
A final remark is that we expect \eqref{eq:psi} holds for any finite mutation type quivers with potential.
\end{remark}

%+++=========================================================================
%+++=========================================================================

\end{document}